\documentclass[10pt]{article}

\usepackage[portrait,margin=2.54cm]{geometry}

\usepackage{amsfonts}
\usepackage{comment}
\usepackage{amssymb}
\usepackage{mathrsfs}
\usepackage{soul}
\usepackage{hyperref}
\usepackage{amssymb,amsthm,amsmath,amsfonts,amsbsy,latexsym}
\usepackage{graphicx}
\usepackage[numeric,initials,nobysame]{amsrefs}
\usepackage{upref,setspace}
\usepackage{enumerate}
\usepackage{paralist}
\usepackage{color}

\usepackage{mathtools}

\definecolor{expcol}{rgb}{1.0,0.5,0.5}
\definecolor{ecol}{rgb}{0.0, 0.5, 1.0}
\definecolor{ab}{rgb}{0.5, 0.0, 1.0}

\numberwithin{figure}{section}
\numberwithin{table}{section}

\numberwithin{equation}{section}

\mathtoolsset{showonlyrefs}

\DeclareMathOperator*{\argmin}{arg\,min}
\newcommand{\R}[0]{\mathbb{R}}

\newcommand{\N}[0]{\mathbb{N}}

\newcommand{\sF}[0]{\mathcal{F}}
\newcommand{\sG}[0]{\mathcal{G}}

\newcommand{\sA}[0]{\mathcal{A}}

\newcommand{\sC}[0]{\mathcal{C}}
\newcommand{\sD}[0]{\mathcal{D}}

\newcommand{\sL}[0]{\mathcal{L}}
\newcommand{\sM}[0]{\mathcal{M}}
\newcommand{\sN}[0]{\mathcal{N}}

\newcommand{\sP}[0]{\mathcal{P}}

\newcommand{\NN}{\mathbb{N}}
\newcommand{\AAA}{\mathbb{A}}

\newcommand{\RR}{\mathbb{R}}

\newcommand{\clp}{\mathcal{P}}
\newcommand{\clc}{\mathcal{C}}

\newcommand{\cld}{\mathcal{D}}

\newcommand{\cls}{\mathcal{S}}

\newcommand{\clm}{\mathcal{M}}
\newcommand{\cla}{\mathcal{A}}
\newcommand{\ch}{\check}

\newcommand{\eps}{\epsilon}

\newcommand{\yr}{\mathfrak{b}^r}

\newcommand{\cln}{\mathcal{N}}

\newcommand{\clb}{\mathcal{B}}

\definecolor{expcol}{rgb}{1.0,0.5,0.5}
\definecolor{ecol}{rgb}{0.0, 0.0, 1.0}

\newtheorem{lemma}{Lemma}[section]
\newtheorem{proposition}[lemma]{Proposition}

\newtheorem{theorem}[lemma]{Theorem}

\newtheorem{condition}[lemma]{Condition}

\newtheorem{remark}[lemma]{Remark}

\newtheorem{definition}{Definition}[section]

\allowdisplaybreaks

\begin{document}

\title{Ergodic Control of Resource Sharing Networks: Lower Bound on Asymptotic Costs.}
\author{Amarjit Budhiraja, Michael Conroy, and Dane Johnson}
%
%

\maketitle

\abstract{Dynamic capacity allocation control for resource sharing networks (RSN) is studied when the networks are in heavy traffic. The goal is to minimize an ergodic cost with a linear holding cost function. Our main result shows that the optimal cost associated with an associated Brownian control problem provides a lower bound for the asymptotic ergodic cost in the RSN for any sequence of control policies. A similar result for an infinite horizon discounted cost has been previously shown for general resource sharing networks in \cite{budcon1} and for general `unitary networks' in \cite{budgho2}. The study of an ergodic cost criterion requires different ideas as bounds on ergodic costs only yield an estimate on the controls in a time-averaged sense which makes the time-rescaling ideas of \cites{budcon1,budgho2} hard to implement. Proofs rely on working with a weaker topology on the space of controlled processes that is more amenable to an analysis for the ergodic cost. As a corollary of the main result, we show that the explicit policies constructed in \cites{budjoh2,budjoh} for general RSN are asymptotically optimal for the ergodic cost when the underlying cost per unit time has certain monotonicity properties.\\ \ \\ 

}

\section{Introduction}
\label{sec:1}

We study capacity allocation control for  a family of  multiserver queueing networks, under heavy traffic, in which some job types may require simultaneous processing by multiple resources (servers). These Resource Sharing Networks (RSNs) were introduced in \cite{masrob} to model Internet flows (see also \cite{verloop2005stability}) however the underlying setting is quite general and is relevant for other communication network applications as well.  We consider a sequence of RSNs indexed by the {\em heavy traffic parameter} $r\in\mathbb{N}$ where all networks in the sequence have $J$ job types and $I$ resources which have individual processing capacities given by the positive vector $C=(C_1, \ldots, C_I)$.  The processing relationship between resources and job types is given by an $I\times J$ `incidence matrix' $K$, where $K_{i,j}=1$ if resource $i$ is required to process job type $j$ and $K_{i,j}=0$ otherwise.  A capacity allocation control policy in this network is given as a time-dependent processing rate assigned to all $J$ job types by the system administrator, $\yr(t)=( \yr_1(t),..., \yr_J(t))$.
The rate assigned to a particular job type must be supplied by all resources needed to process that job type as determined by the matrix $K$, which means that an admissible control must satisfy the capacity constraint $K\yr(t) \leq C$ for all $t\ge 0$.  For each job type, arrivals are given by a renewal process and job sizes (which determine the integrated processing rate required for completion) are iid nonnegative random variables.  We assume mutual independence between arrivals and sizes and from job type to job type.  Although the structure of the networks, as characterized by $I$, $J$, $K$, and $C$, does not depend on $r$ the distributions of job arrivals and sizes do in such a way that a suitable heavy traffic assumption (see Condition \ref{heavytraffic}) is satisfied.  A given strictly positive vector $h=(h_1,\ldots, h_J)$ determines a linear holding cost function on the queue length and we study the  long-term cost per unit time (also referred to as the ergodic cost), defined in (\ref{eq:ergcoscri}), for the  diffusion-scaled queuing system,  as $r\rightarrow\infty$.

Minimizing a given cost for a RSN over all potential control policies is, in general, a very difficult problem.  However, in many situations this optimization problem can be formally approximated by a more tractable so-called Brownian control problem (BCP) related to the RSN. BCPs were introduced in \cite{harrison1988brownian} and are formal diffusion approximations to the network control problem in the heavy-traffic limit.  While the BCP already represents a dramatic simplification of the control problem, it was shown in \cite{harvan} that under suitable conditions, the BCP can be restated in terms of a dimensionally-reduced equivalent workload formulation (EWF).  In the EWF the state process changes from the $J$-dimensional queue length to the $I$-dimensional workload for each resource which is computed using the matrix $K$ and the expected job sizes.  In applications where $I \ll J$, this represents a significant dimensional reduction in the control problem to be solved.  To better understand the motivation for the EWF see \cite{harvan} for a discussion of how a particular workload can be thought of as an equivalence class of queue lengths which differ from one another by a reversible control.

For our main result, we show that the optimal cost associated with the BCP in the EWF provides a lower bound for the asymptotic ergodic cost in the RSN for any sequence of  control policies that are admissible (Definition \ref{def:admissible}).
 This is the same type of result shown in \cite{budcon1} for an infinite-horizon discounted cost criterion and in \cite{budgho2} for discounted cost associated with `unitary' networks. To the best of our knowledge there are no analogous results known for an ergodic cost criterion.  Our result fits into a series of works that investigate the relationship between {\em asymptotic optimality} and BCP for various network models under heavy traffic assumptions (cf. \cite{marshrson,atakum,belwil1,belwil2,atar2014asymptotic,budgho1,DaiLin} for a small sample of works on this theme). A sequence of control policies is said to be asymptotically optimal if the limit of the corresponding costs as $r\rightarrow\infty$ is minimal, and since from our results the BCP provides a lower bound
 for this asymptotic minimal cost, any sequence of controls whose costs converge to this value is asymptotically optimal.

While not considered here, it is an interesting and challenging problem to design control policies that are easy-to-implement and achieve a certain quality of asymptotic performance.  Since the BCP provides a lower bound on the minimal asymptotic cost it may provide some helpful intuition for this pursuit.  In recent works \cite{budjoh,budjoh2}, explicit policies are constructed for RSNs which achieve what is known as `hierarchical greedy ideal' (HGI) asymptotic performance \cite{harmandhayan} for the ergodic cost considered here as well as for a discounted cost. If the cost per unit time as a function of the workload in the EWF (see \eqref{eq:hhat}) is nondecreasing, the BCP has an explicit solution given as a reflected Brownian motion,  and the asymptotic cost associated with these policies matches the cost associated with this reflected Brownian motion, namely the lower bound we establish. Thus, as discussed further in Remark \ref{rem:rem1}, this work proves asymptotic optimality of the policies given in \cite{budjoh,budjoh2}, with an ergodic cost criterion, when the above monotonicity property is satisfied.
  In general, however, HGI policies may not be optimal for the BCP.

This paper is organized as follows. Below, we discuss assumptions, proof techniques, and notation. Section \ref{sec:netcont} gives basic definitions and presents the control problem of interest. Section \ref{sec:2} describes the BCP used in our main result (Section \ref{sec:Main}). Section \ref{sec:prelim} presents auxiliary results used in the the proof of the main theorem, which is proved in Section \ref{sec:erg}. We conclude with the proofs of various technical results in Section \ref{sec:MovedProofs} which were postponed in order to streamline the presentation.

\subsection{Discussion of assumptions and proof techniques} 

We now  comment on the assumptions made.  We make standard mutual independence, positivity, and uniform square intergrability assumptions for interarrival times and job sizes. Additionally, the heavy traffic assumption (Condition \ref{heavytraffic}) we make is standard, cf. \cite{harmandhayan}.  Roughly, it says that parameters converge appropriately and  that the resource capacity equals resource demand, asymptotically. One key assumption is that of `local traffic' on each resource, which was first introduced in \cite{kankelleewil} (and also used in several subsequent works, see e.g. \cite{harmandhayan}). This says that each resource in the RSN  has a unique corresponding job type that only uses that resource. This condition ensures that the state space of the workload process (see \eqref{eq:workloadprocess}) is the entirety of $[0,\infty)^I$. These assumptions are presented precisely in Section \ref{sec.assump}. 

In  \cite{budcon1}, the analogous result to our main result, namely Theorem \ref{thm:main},  is established for an infinite-horizon discounted cost criterion. This work  also briefly discusses an ergodic cost criterion however no proofs are provided and its treatment is left as an open problem.  
A key step in the proof of Theorem \ref{thm:main} is establishing suitable tightness properties. As one would expect, the desired tightness properties stem from the boundedness of costs associated with a given sequence of controls. 
Indeed, this observation was the key ingredient in proofs of \cite{budcon1} which allowed establishing desired tightness results using a natural time-rescaling idea.
 However, in contrast to the discounted cost case, bounds on the ergodic cost  only guarantee `good behavior' of the controls in a time-averaged sense.
This  makes the time-change approach for establishing tightness taken in \cite{budcon1} hard to implement.
Instead we  take a different approach to proving required tightness properties.  Specifically, instead of arguing tightness of the controls as random variables  in the Skorokhod path space, we establish  tightness in the coarser topology of weak convergence (see Theorem \ref{thm:tightness}). 
More precisely, we establish tightness of certain random path occupation measures (see Definition 
\ref{def:occMeas}) which view the control process paths as  measures on $[0,\infty)$ rather than as elements of the Skorohod path space. This weak tightness property, although insufficient to prove the convergence of controlled processes,  turns out to be enough to argue convergence of the costs and for characterization of the weak limit points of the occupation measures. This is the main idea in the proof.
  We remark that   this  approach of using  tightness of the controls in the topology of weak convergence can also provide an alternative proof for the discounted cost case treated in  \cite{budcon1} that does not require the time-rescaling arguments used there.

An important result in the proof of Theorem \ref{thm:main} is Theorem \ref{thm:conv}, which 
gives a key characterization of the weak limit points of our random path occupation measures. In particular it shows that any limit point can be identified as the law of a controlled system in the appropriate BCP. Its proof rests on two results --Propositions \ref{prop:ArrServBnds} and \ref{prop:FuildScaledAllocConv}. The second parts  of these two results establish that the long-term time-average of residual service times and the long-term time-average of the difference between an arbitrary sequence of fluid scaled controls and the nominal capacity allocation determined from the heavy traffic condition, both vanish. This is another new ingredient in comparison to the proofs in \cite{budcon1}, where due to a discounted cost criterion such a careful  analysis is not needed. Arguments therein are closer to the uniform-in-time estimates in Propositions \ref{prop:ArrServBnds} (i) and \ref{prop:FuildScaledAllocConv} (i), which  follow from basic properties of renewal processes. Proofs of statements involving the arrival processes (namely part (i) of Propositions \ref{prop:ArrServBnds} and  \ref{prop:FuildScaledAllocConv}) are simpler than those for the service process (namely parts (ii) of the two propositions) due to the form of control in the network. That is, a system administrator can influence resource allocation, but has no influence on arrivals to the network (cf. \eqref{eq:Qdef}). The proofs of Propositions \ref{prop:ArrServBnds} and \ref{prop:FuildScaledAllocConv} are rather technical, so they are saved for the end (Section \ref{sec:MovedProofs}).

\subsection{Notation and conventions.} 
We denote $\R_+ = [0,\infty)$. For $d \in \N$, let $\sD^d= \sD([0,\infty) : \R^d)$ (resp. $\sD_+^d = \sD([0,\infty) : \R_+^d)$) denote the space of functions that are right continuous with left limits (RCLL) from $[0, \infty)$ to $\R^d$ (resp. $\R_+^d$) equipped with the usual Skorokhod topology. 
Also, let  $\sC^d= \sC([0,\infty) : \R^d)$ (resp. $\sC_+^d = \sC([0,\infty) : \R_+^d)$) denote the space of continuous functions from
$[0, \infty)$ to $\R^d$ (resp. $\R_+^d$) equipped with the local uniform topology. For fixed $T > 0$, the spaces $\sD([0,T] : \R^d)$, $\sC([0,T] : \R^d)$, $\sD([0,T] : \R_+^d)$, and $\sC([0,T] : \R_+^d)$ are defined similarly.  All stochastic processes in this work will have sample paths that are RCLL unless noted explicitly.  We let $\sN_+^d$ denote the nondecreasing functions in $\sD_+^d$ which is the space that contains the controls we will be working with.  
Let $\sM$ denote the space of nonegative, locally finite measures on $[0,\infty)$.
Recall that, for $m_n, m \in \sM$ we say that $m_n\to m$ in the {\em vague topology}, if for every continuous function $f: [0,\infty) \to \R$ with compact support, $\int f dm_n \to \int f dm$.
We let $\sM^d$ denote the $d$-fold product space with the usual product topology, where $\sM$ is equipped with the topology of vague convergence of measures.  Note that we can identify elements of $\sN^d$ and $\sM^d$, but the Skorokhod topology associated with $\sN^d$ is finer than the vague topology associated with $\sM^d$.  For our purposes here we find it more convenient to work with the coarser weak convergence topology associated with $\sM^d$ when proving tightness and convergence of controls.  Details about this are provided in Section \ref{sec:measSpaceDet}.

We let $\iota\in\sC^1$ denote the identity function on $\mathbb{R}_+$, namely $\iota(t) = t$. For $m \in \N$, we denote by $\AAA_m$ the set $\{1, 2, \ldots, m\}$.  For any Polish space $\mathcal{X}$ we let $\mathcal{B}(\mathcal{X})$ denote the Borel sets, $\mathcal{P}(\mathcal{X})$ denote the space of probability measures on $\mathcal{X}$, $\sC_b(\mathcal{X})$ the space of continuous, bounded, real-valued functions on $\mathcal{X}$, and $\mathcal{C}_0(\mathcal{X})$ the space of continuous, real-valued functions on $\mathcal{X}$ with compact support.  We let $\sC_0^2(\mathbb{R}^d)$ denote the space of continuous, real-valued functions on $\mathbb{R}^d$ with compact support and continuous first and second derivatives.  For  $v\in\mathbb{R}^{d}$, $|v|$ denotes its Euclidean norm. 
We will use coordinate-wise inequalities on vectors, e.g.  for  $u,v\in\mathbb{R}^{d}$ and $c\in\mathbb{R}$ the statements $u\ge v$
and $u\geq c$ mean  $u_{j}\geq v_{j}$  and $u_{j}\geq c$ for all $j\in\AAA_{d}$, respectfully.  
 As a convention, for a real sequence $\{a_l\}_{l\in \NN}$, $\sum_{l=1}^{n} a_l$ is taken to be $0$ if $n=0$.
We use $\overset{P}{\to}$ to denote convergence in probability and $\Rightarrow$ to denote convergence in distribution. $\mathcal{I}_A$ denotes the indicator function of the set $A$. A sequence $\{\theta_n, n\ge 1\}$ of probability measures on $\cld^d$ is said to be $\clc$-tight if it is tight in the usual Skorohod topology and any weak limit point $\theta$ is supported on $\clc^d$.



\section{The network and control policies}
\label{sec:netcont}

A RSN in heavy traffic is a sequence of stochastic processing networks, indexed by a traffic parameter $r \in \NN$, each with $J$ types of jobs and $I$ resources for processing them.  All networks in the sequence are described by a common $I\times J$ incidence matrix $K$ such that $K_{i,j}=1$ if the $j$th job type requires service from the $i$-th resource and $K_{i,j}=0$ otherwise.  For any job type $j$ the set  $\{i:K_{i,j}=1\}$ is nonempty and job type $j$ will be processed by all resources in this set simultaneously. In particular, all resources in this set will allocate the same capacity to  job type $j$ at every instant.  The capacity for each resource $i \in \AAA_I$ is denoted by $C_i$. This means that if at any time instant work of type $j \in \AAA_J$ is being processed at rate $\yr_j$ then we must have $C \ge K\yr$, where $C = (C_1, \ldots, C_I)$.  As the traffic parameter $r$ goes to $\infty$ the networks approach criticality in the sense that the traffic intensity converges to 1 (this is made precise in Condition \ref{heavytraffic}).  

For job type $j \in \AAA_J$, let $\{u_j^r(k)\}_{k\in \N}$ be the iid interarrival times and $\{v_j^r(k)\}_{k\in \N}$ be the associated iid job-sizes. For each $r$, the random variables in the collection $\{u_j^r(k), v_j^r(k), k \in \NN, j \in \AAA_J\}$
are taken to be mutually independent.  We assume that these random variables have finite second moment and let $\alpha_j^r = 1/E[u_j^r(1)]$ and $\beta_j^r = 1/E[v_j^r(1)]$ denote the rates of the interarrival times and job sizes. Also let $\sigma^{u,r}_j$ and $\sigma^{v,r}_j$ denote the standard deviations of $u_j^r(1)$ and $v_j^r(1)$, respectively. We assume a first in, first out (FIFO) policy, meaning that for each job type the oldest job in the queue is processed before another one is started.  Note that if the job sizes are exponentially distributed the `memoryless' property implies that the way the processing rate is distributed among jobs in a particular queue has no impact on the queue length distribution so in that case we can drop FIFO assumption without impacting the results.  We now introduce our main assumptions.

\subsection{Assumptions}\label{sec.assump}
In addition to the  assumptions of mutual independence among interarrival times and job sizes discussed above, we assume the following. 

\begin{condition}\label{cond:uniformintegrability} $P(u_j^r(1) > 0 ) = P(v_j^r(1) > 0) = 1$ for all $r \in \N$ and $j \in \AAA_J$. Furthermore, for each $j$, $\{u_j^r(1)^2\}_{r\in\N}$ and $\{v_j^r(1)^2\}_{r\in\N}$ are uniformly integrable. 
\end{condition}
%
In fact, for notational convenience (and without loss of generality) we will assume that $u_j^r(l)>0$ and $v_j^r(l)>0$ for all $r,l\in\mathbb{N}$ and $j\in\AAA_J$.

Let $\rho^r_j = \alpha^r_j/\beta^r_j$ for all $j\in \AAA_J$ and $\rho^r = (\rho^r_1, \ldots, \rho^r_J)$. The following will be our main heavy traffic condition: 

\begin{condition}\label{heavytraffic} (Heavy Traffic.)
	For each $j \in \AAA_J$, there exist $\alpha_j, \beta_j \in (0, \infty)$ and $\bar \alpha_j, \bar \beta_j \in \RR $ such that 
	$\lim_{r\to\infty} r(\alpha^r_j - \alpha_j) = \bar \alpha_j$, $\lim_{r\to\infty} r(\beta^r_j - \beta_j) = \bar \beta_j$, 
and there exist $\sigma^u_j, \sigma^v_j \in (0, \infty)$ such that 
$\lim_{r\to\infty} \sigma^{u,r}_j = \sigma^u_j$, $\lim_{r\to\infty} \sigma^{v,r}_j = \sigma^v_j$. 

Furthermore, with $\rho_j = \alpha_j/\beta_j$ for each $j\in \AAA_J$ and $\rho = (\rho_1, \ldots, \rho_J)$,
$C = K\rho$.

\end{condition}
Define $\theta\doteq K\eta$, where
\begin{equation}
\label{eq:etaTheta}
\eta_j \doteq \frac{\bar \alpha_j \beta_j - \alpha_j \bar \beta_j}{\beta_j^2}, \qquad j\in\AAA_J,
\end{equation}
and $\eta=(\eta_1,\ldots,\eta_J)$ and note that Condition \ref{heavytraffic} implies that
\[
	\lim_{r\to\infty} r(\rho^r-\rho) = \eta.
\] 
The quantity $\theta$ arises when considering the workload process given below in \eqref{eq:workloadprocess} and plays an important role in the Brownian control problems introduced in Section \ref{sec:2}.  Throughout we assume that $\theta <0$. This is a necessary and sufficient condition for the reflected Brownian motion that arises from HGI performance in the Brownian control problem to have a unique stationary distribution (see \cite{harwil1}). 
As we will see in Proposition \ref{prop:thetaNegFinCost} and Theorem \ref{thm:monotsoln}, $\theta<0$ is also necessary and sufficient for a finite ergodic cost.

We will assume the following local traffic condition which says that every resources has at least one associated job-type which only requires processing from that particular resource.  This condition guarantees that the workload process (see \eqref{eq:workloadprocess}) can achieve all vectors in $\mathbb{R}_+^I$.
\begin{condition} (Local Traffic.)
\label{cond:locTraffic}
For every $i\in\AAA_I$ there exists $j\in\AAA_J$ such that $K_{i,j}=1$ and $K_{l,j}=0$ for all $l\in\AAA_I \setminus \{ i\}$.
\end{condition}

Conditions \ref{cond:uniformintegrability} --\ref{cond:locTraffic} will be assumed throughout and will not be noted explicitly in the statements of various results.
\subsection{State processes}
\label{sec:stateProc}

Define the collection of renewal processes 
\[
	A_j^r(t) = \max\left\{\ell \in \N : \sum_{i=1}^\ell u_j^r(i) \le t \right\}, \qquad j \in \AAA_J,\; t \ge 0, 
\]
and 
\[
	S_j^r(t) = \max \left\{ \ell\in \N : \sum_{i=1}^\ell v_j^r(i) \le t \right\}, \qquad j \in \AAA_J,\; t \ge 0. 
\]

A control policy in the $r$-th network is an $\mathbb{R}_+^J$-valued stochastic process $B^r(\cdot)$ satisfying certain feasability and measurability conditions described in Definition \ref{def:admissible}.  The quantity $B^r_j(t)$ represents the cumulative amount of capacity allocated to type $j$ jobs at time $t$.  
We denote by $Q^r_j(t)$, $j \in \AAA_J$, the number of type-$j$ jobs in the queue at time instant $t$.  The $J$-dimensional queue length process is given by
\begin{equation}
\label{eq:Qdef}
	Q^r(t) \doteq q^r + A^r(t) - S^r(B^r(t)), \qquad t\ge 0,
\end{equation}
where $q^r \in \N^J$ denotes the initial queue length vector, $A^r(t)$ is the number of jobs that have arrived by time $t$, and $S^r(B^r(t))$ is the number of jobs that have been processed by time $t$ under the allocation policy $B^r(\cdot)$. Here, we are using the notational shorthand 
\[
	S^r(B^r(t)) =  (S_1^r(B_1^r(t)), S_2^r(B_2^r(t)), \ldots, S_J^r(B_J^r(t))). 
\]
 Recall that $C_i$ indicates the capacity of resource $i$.  We define the $I$-dimensional process 
\begin{equation}
\label{eq:Udef}
	U^r(t) \doteq tC - KB^r(t), 
\end{equation}
which gives the unused capacity of each resource at time $t$.  Let $M^r$ denote the $J\times J$ diagonal matrix with entries $\{1/\beta^r_j\}_{j\in\AAA_J}$, and let $M$ denote the $J\times J$ diagonal matrix with entries $\{1/\beta_j\}_{j\in\AAA_J}$.  The $I$-dimensional workload process $W^r$, which gives the amount of work in the system for each resource, is defined by 
\begin{equation}\label{eq:workloadprocess}
	W^r(t) \doteq KM^rQ^r(t) = w^r+KM^rA^r(t)-KM^rS^r(B^r(t)), \qquad  t \geq 0
\end{equation}
where $w^r\doteq KM^r q^r$.  

 In order to study the behavior as the systems approach criticality we consider two types of scaling: diffusion scaling and fluid scaling. In both of these scalings,
time is accelerated by a factor of $r^2$ but in diffusion scaling the magnitude is scaled down by a factor of $r$ while in the fluid scaling the magnitude is scaled down by a factor of $r^2$.  
Processes using the diffusion scaling will be denoted with a `hat' symbol while processes with the fluid scaling will be denoted with a `bar' symbol.  We define
\begin{align*}
\bar{A}^r(t) &= r^{-2}A^r(r^2t), & \bar{S}^r(t) &= r^{-2}S^r(r^2t), \\
\bar{Q}^r(t) &= r^{-2}Q^r(r^2t), & \bar{W}^r(t) &= r^{-2}W^r(r^2t), \\
\bar{B}^r(t) &= r^{-2}B^r(r^2t), & \bar{U}^r(t) &= r^{-2}U^r(r^2t), 
\end{align*}
and
\begin{align*}
\hat{A}^r(t) &= r^{-1}\left(A^r(r^2t) - r^2t \alpha^r\right), &\hat{S}^r(t) &= r^{-1}\left(S^r(r^2t) - r^2t\beta^r\right), \\
\hat{Q}^r(t) &= r^{-1}Q^r(r^2t), & \hat{W}^r(t) &= r^{-1}W^r(r^2t),\\
\hat{B}^r(t) &= r^{-1}B^r(r^2t), & \hat{U}^r(t) &= r^{-1}U^r(r^2t). 
\end{align*}
Our primary focus is on the diffusion scaling so for notational convenience we also define the diffusion scaled initial queue lengths and workloads, $\hat{q}^r\doteq \frac{1}{r}q^r$ and $\hat{w}^r\doteq KM^r\hat{q}^r$, as well as the process
\begin{equation}
\label{eq:XhatDef}
\hat{X}^r(t)\doteq KM^{r}\left( \hat{A}^{r}(t)-\hat{S}^{r}(\bar{B}^{r}(t))\right) +rtK(\rho
^{r}-\rho).
\end{equation}
The process $\hat X^r$ may be viewed as a `pre-Brownian motion' as we will see that, under conditions, its weak limit is given as a Brownian motion (with a drift).
Observe that the diffusion scaled workload process can be written as
\begin{align}
\label{eq:stateq1}
\hat{W}^{r}(t) &=KM^{r}\hat{q}^{r}+KM^{r}\left( \hat{A}^{r}(t)-\hat{S}^{r}(\bar{B}
^{r}(t))+r(\alpha^{r}t-\beta^{r}\bar{B}^{r}(t))\right)  \nonumber\\
&=KM^{r}\hat{q}^{r}+KM^{r}\left( \hat{A}^{r}(t)-\hat{S}^{r}(\bar{B}^{r}(t))\right) +Krt\rho
^{r}-rK\bar{B}^{r}(t) \nonumber\\
&=KM^{r}\hat{q}^{r}+KM^{r}\left( \hat{A}^{r}(t)-\hat{S}^{r}(\bar{B}^{r}(t))\right) +rtK(\rho
^{r}-\rho)+r(Ct-K\bar{B}^{r}(t))\nonumber\\
&= \hat{w}^r+\hat{X}^r(t)+\hat{U}^r(t)
\end{align}

\noindent where in the third equality we used $C=K\rho$ which comes from Condition \ref{heavytraffic}.

\subsection{Admissable control policies}

We will now make precise what one means by an admissible control policy.
We need to introduce the following multiparamater filtrations.
\begin{definition} For $m = (m_1, \ldots, m_J)\in\N^J$ and $n = (n_1, \ldots, n_J) \in \N^J$, let 
\[
	\sF^r(m, n) = \sigma\left\{ u_j^r(m'_j), v_j^r(n'_j) : m'_j \le m_j, n'_j\le n_j, j\in \AAA_J  \right\}.
\]
Then $\{\sF^r(m,n), m, n\in \N^J\}$ is a multiparameter filtration generated by the interarrival times and job-sizes with the following partial ordering: 
\[
	(m^1, n^1) \le (m^2, n^2) \;\;\text{if and only if}\;\; m_j^1\le m_j^2\;\;\text{and}\;\; n_j^1\le n_j^2\;\;\text{for all $j\in\AAA_J$}.
\]
Let 
\[
	\sF^r = \sigma\left\{ \bigcup_{(m,n)\in \N^{2J}} \sF^r(m,n) \right\}. 
\]
\end{definition}
We refer the reader to \cite[Section 6.2]{ethier2009markov} for definitions and properties of multiparameter stopping times.
We can now define the class of admissible control policies. Parts (ii)-(iii) are natural feasibility constraints while parts (iv)-(v) capture the fact that the control policies can only be based on information available until the current time instant. These are satisfied for a very broad family of natural control policies (cf.   \cite[Theorem 5.4]{budgho2}).

\begin{definition}\label{def:admissible} For $r \in \NN$, a $\mathbb{R}_+^J$-valued process $B^r$ is said to be an {\em admissible resource allocation policy} or an {\em admissible control policy} if it satisfies the following:
	\begin{enumerate}[(i)]
	\item $t\mapsto B^r(t)$ is an absolutely continuous, nonnegative, nondecreasing function on $[0,\infty)$, with $B^r(0) = 0$.
	\item $C \ge K\frac{d}{dt} B^r(t)$ for a.e. $t\ge 0$, a.s.
	\item The process $Q^r(\cdot)$ defined by \eqref{eq:Qdef} satisfies $Q^r(t) \ge 0$ for all $t\ge 0$ a.s.
	\item For each $r\in \N$ and $t\ge 0$ consider the $\mathbb{N}^{2J}$ valued random variable
	\begin{equation*}
	\tau^{r}(t)\doteq (\tau^{r,A}_{1}(t),\ldots,\tau^{r,A}_{J}(t),\tau^{r,S}_{1}(t),\ldots,\tau^{r,S}_{J}(t))
	\end{equation*}
	where, for  $j\in\mathbb{N}^{J}$,
	\begin{equation*}
	 \tau^{r,A}_{j}(t)=\min\left\{n\geq 0: \sum_{l=1}^{n}u^{r}_{j}(l)\geq r^{2}t\right\}, 
	\end{equation*}
	and
	\begin{equation*}
	 \tau^{r,S}_{j}(t)=\min\left\{n\geq 0: \sum_{l=1}^{n}v^{r}_{j}(l)\geq r^{2}\bar{B}^{r}_{j}(t)\right\}.
	\end{equation*}
	 Then $\tau^{r}(t)$ is a (multiparameter) $\{\mathcal{F}^{r}(n,m)\}$-stopping time for all $t\geq 0$.
	\item If we define the filtration 
	\begin{align*}
		\sG^r(t) &= \sF^r(\tau^r(t)) \\
			    &= \sigma\left\{ A\in \sF^r : A\cap \{\tau^r(t) \le (m,n)\} \in \sF^r(m,n)\;\;\text{for all}\;\;(m,n) \in \N^{2J}\right\}, 
	\end{align*}
	then $B^r(r^2t)$ is $\{\sG^r(t)\}$-adapted. 
	\end{enumerate}
	The collection of admissible controls for the $r$th network is denoted $\mathcal{A}^r$.  A sequence of control policies $\{B^r\}_{r\in\mathbb{N}}$ is called admissible  if, for each $r\in\mathbb{N}$, $B^r\in \mathcal{A}^r$. Denote the class of all such admissible sequences  as $\mathcal{A}$. 
\end{definition}

We note that the requirement in Definition \ref{def:admissible} part (ii) implies that for any admissible control policy $B^r$, the process $U^r(t) = tC - KB^r(t)$ is nonnegative, nondecreasing, and for any $0\leq s\le t$, 
\[ 
	0 \le K(B^r(t) - B^r(s)) = (t-s)C - (U^r(t) - U^r(s)) \le (t-s)C. 
\]
Recall that for each $j\in\AAA_J$ there is an $i\in\AAA_I$ such that $K_{ij} = 1$. Thus, for all $j\in\AAA_J$ and $0\leq s\leq t$ we have
\begin{equation}\label{eq:CLip}
	0\le B^r_j(t) - B^r_j(s) \le L(t-s), \qquad L = \max_{1\le i \le I}C_i, 
\end{equation}
i.e.,  $B^r$ is Lipschitz continuous with Lipschitz constant not depending on $r$. 

For the rest of this work we will only consider admissible control policies and so frequently the adjective `admissible' will be omitted.

\subsection{Cost function}

We now introduce the cost function of interest. We consider linear `holding costs' given through a fixed, strictly positive $J$-dimensional vector $h$.  In the $r$-th network, consider the long-term cost per unit time (or the ergodic cost) associated with a control policy $B^r$, defined as
\begin{equation}
	\label{eq:ergcoscri}
	J^r_E(B^r) = \limsup_{T\to \infty}  E \left[ \frac{1}{T}\int_0^T  h \cdot \hat{Q}^r(t) \,dt \right]. 
\end{equation}
For a sequence of control policies $\{B^r\}_{r\in\mathbb{N}}$, the associated  ergodic asymptotic cost is defined as 
\begin{equation}\label{asymptoticcost}
	J_E(\{B^r\}) = \liminf_{r\to\infty} J^r_E(B^r). 
\end{equation}
The infimum of the asymptotic ergodic cost over all admissible sequences of control policies will be referred to as
the {\em asymptotic value function} for the ergodic control problem and is given as 
\begin{equation}
\label{eq:asymValueFcn}
	 J^*_E = \inf_{\{B^r\} \in \sA} J_E(\{B^r\}). 
\end{equation}

\section{The Brownian control problem and its workload formulation}
\label{sec:2}
The main result of this work  gives a lower bound on the asymptotic ergodic control cost in terms of the 
value function of a certain control problem for Brownian motions \cite{har1, har2, harvan}. We  present below the {\em Equivalent Workload Formulations} (EWF) of these control problems. We refer the reader to \cite{harvan} for a discussion on equivalence between this formulation and the Brownian control problems as formulated in Harrison \cite{harrison1988brownian}.
We begin by introducing the notion of an effective cost function. 
For each $w\in \RR_+^I$, define the effective cost function as
\begin{equation}\label{eq:hhat}
	\hat{h}(w) = \min\{h \cdot q : KMq = w, q\ge 0\}. 
\end{equation}
Note that, from the local traffic condition (Condition \ref{cond:locTraffic}), the set on the right side is nonempty for every $w \in \RR_+^I$.
It is known that we can select a continuous minimizer in the above linear program (cf. \cite{boh1}), i.e. there is a continuous map $q^*:\RR_+^I \to \R_+^J$ such that 
\[
	q^*(w) \in \argmin_q\{h \cdot q : KMq = w, q\ge 0\}.
\]
Recall 
$\theta \doteq K \eta$ from \eqref{eq:etaTheta}
and let $\Sigma$ denote the $I \times I$ matrix  
\begin{equation}\label{eq:Sigma}
	\Sigma\doteq KM(\Sigma^u+\Sigma^v R)^{1/2} (KM(\Sigma^u+\Sigma^v R)^{1/2})^T, 
\end{equation}
where $\Sigma^u$ is the $J\times J$ diagonal matrix with entries $\alpha_j^3(\sigma^u_j)^2$, $\Sigma^v$ is the $J \times J$ diagonal matrix with entires $\beta_j^3(\sigma^v_j)^2$, and $R$ is the $J \times J$ diagonal matrix with entries $\rho_j$.  
The EWF and the associated controls and state processes are defined as follows. 
%

\begin{definition} 
	\label{def:BCP2}
Let $(\tilde{\Omega}, \tilde{\sF}, \tilde{P}, \{\tilde{\sF}(t)\})$ be a filtered probability space which supports an $I$-dimensional $\tilde{\sF}(t)$-Brownian motion $\tilde{X}$ with drift $\theta$ and covariance matrix $\Sigma$.
	An $I$-dimensional $\{\tilde{\sF}(t)\}$-adapted process $\tilde{U}$  on this space is called an admissible control for the ergodic EWF if there is a $\{\tilde{\sF}(t)\}$-adapted $\RR_+^I$-valued process $\tilde{W}$ such that the following hold $\tilde{P}$-a.s.: 
	\begin{enumerate}[(i)]
	\item $\tilde{W}(t) = \tilde{W}(0) +\tilde{X}(t) +  \tilde{U}(t)$ for all $t\ge 0$, 
	\item $\tilde{U}(t)$ is nondecreasing and $\tilde{U}(0) = 0$, 
	\item $\tilde W$ is stationary, namely, for all $t \ge 0$,
$\tilde{W}(t+\cdot)$ has the same distribution on $\cld^I_+$ as
$\tilde{W}(\cdot)$.
	\end{enumerate}
Denote the class of all such admissible controls as $\tilde{\sA}_E$. 
\end{definition}
The cost for a control $\tilde{U} \in \tilde{\sA}_E$ in the ergodic cost Brownian control problem (BCP) is defined as
\begin{equation}
	\tilde{J}^{BCP}_{E}(\tilde{U}) =  \tilde{E}[\hat h (\tilde{W}(0))], 
\end{equation}
where $\tilde{E}$ denotes expectation on $(\tilde{\Omega}, \tilde{\mathcal{F}}, \tilde{P})$. 
The corresponding value function is
\begin{equation}
\label{eq:valueFcnDef}
\tilde{J}^{BCP,*}_{E} \doteq
\begin{cases}
	\infty \text{ if }\tilde{\sA}_E=\emptyset\\
	 \inf_{\tilde{U}\in \tilde{\sA}_E} \tilde{J}^{BCP}_E(\tilde{U})\text{ otherwise,}
\end{cases}
\end{equation}
where the infimum above is also taken over all filtered probability spaces as in Definition \ref{def:BCP2}.
The following proposition shows that $\theta<0$ is a necessary condition for $\tilde{J}^{BCP,*}_E<\infty$. The sufficiency of this condition is a consequence of Theorem \ref{thm:monotsoln} given below.
Before presenting this result we recall the definition of the Skorokhod map.
\begin{definition}
	\label{def:smsp}
	Let $\psi \in \mathcal{D}^I$ such that $\psi(0)\in \RR^I_+$. The pair $(\varphi,\eta) \in \mathcal{D}^{2I}$ is said to solve the {\em Skorokhod problem} for $\psi$ (in $\RR^I_+$, with normal reflection)
	if $\varphi = \psi+\eta$; $\varphi(t)\in \RR^I_+$ for all $t \ge 0$; $\eta(0)=0$; $\eta$ is nondecreasing and $\int_{[0,\infty)} 1_{\{\varphi_i(t)>0\}} \,d\eta_i(t) =0$ for all $i \in \NN_I$.
	We write $\varphi = \Gamma(\psi)$ and refer to $\Gamma$ as the $I$-dimensional {\em Skorokhod map}.
\end{definition}
It is  known that there is a unique solution to the above Skorokhod problem for every $\psi \in \mathcal{D}^I$ with $\psi(0)\in \RR^I_+$.  When $I=1$, we will denote $\Gamma$ as $\Gamma_1$.
\begin{proposition}
\label{prop:thetaNegFinCost}
If $\theta_i\geq 0$ for some $i\in \AAA_I$ then  $\tilde{J}^{BCP,*}_E=\infty$
\end{proposition}
\begin{proof}
Let $i\in\AAA_I$ satisfy $\theta_i\geq 0$ and assume $\tilde{\sA}_E\neq \emptyset$.  Let $\tilde{U}\in \tilde{\sA}_E$ be an arbitrary admissible control on a filtered probability $(\tilde{\Omega}, \tilde{\sF}, \tilde{P}, \{\tilde{\sF}(t)\})$ with $\tilde{X}$ and $\tilde{W}$ as in Definition \ref{def:BCP2}.  Then from well known minimality properties of the Skorokhod map (cf. Proposition 5.1 in \cite{budjoh2}) we have
\begin{equation*}
\tilde{W}_i(t)=\tilde{W}_i(0)+\tilde{X}_i(t)+\tilde{U}_i(t)\geq \Gamma_1\left(\tilde{W}_i(0)+\tilde{X}_i(t)\right).
\end{equation*}
Since $\tilde{X}_i(\cdot)$ is a one-dimensional Brownian motion with drift $\theta_i\geq 0$, we have, for any $w\in \mathbb{R}_+$ and 
$K<\infty$
\begin{equation*}
\lim_{t\rightarrow\infty}P\left( \Gamma_1\left(w+\tilde{X}_i(\cdot)\right)(t) \geq K\right)=1.
\end{equation*}
Since for all $w_1,w_2\in \mathbb{R}_+$ and $t\geq 0$ we have
\begin{equation*}
\sup_{s\in [0,t]}\bigg| \Gamma_1\left(w_1+\tilde{X}_i(\cdot)\right)(t)- \Gamma_1\left(w_2+\tilde{X}_i(\cdot)\right)(t) \bigg|\leq 2 |w_1-w_2|
\end{equation*}
and $\tilde{X}(\cdot)$ is independent of $\tilde{W}(0)$ it follows that
\begin{equation*}
\lim_{t\rightarrow\infty}P\left( \tilde{W}_i(t) \geq K\right)\geq\lim_{t\rightarrow\infty}P\left( \Gamma_1\left(\tilde{W}_i(0)+\tilde{X}_i(\cdot)\right)(t) \geq K\right)=1
\end{equation*}
for all $K<\infty$.  This contradicts the that fact $\tilde{W}(t)$ has the same distribution as $\tilde{W}(0)$ for all $t\geq 0$. The result follows.
\end{proof}
Obtaining explicit simple form solutions for the control problems in Definition \ref{def:BCP2} is in general impossible.
However, there is one important setting, given in Theorem \ref{thm:monotsoln} below, where explicit solutions are available.  Proof  of the first statement follows from \cite{harwil1}, while the second statement is a consequence of the well known exponential integrability of the stationary distribution (cf. \cite{budlee2}). Final statement in the theorem follows from standard minimality properties of the Skorokhod map with normal reflections on the domain $\RR_+^I$ (cf. \cite{harwil1}).  Proof is omitted.
\begin{theorem}
	\label{thm:monotsoln}
	Let $(\tilde{\Omega}, \tilde{\sF}, \tilde{P}, \{\tilde{\sF}(t)\})$ and $\tilde{X}$ be as in Definition \ref{def:BCP2}.
	Suppose in addition that $\theta <0$. Then there is a unique stationary distribution $\pi$ for the Markov process described by 
	\begin{equation}\label{eq:wstar}
		\ch W(t) \doteq \Gamma({w} + \tilde X)(t),
		\qquad t \ge 0, \; w \in \R_+^I.
		\end{equation}
	Assume without loss of generality that the filtered probability space $(\tilde{\Omega}, \tilde{\sF}, \tilde{P}, \{\tilde{\sF}(t)\})$ supports an $\tilde{\sF}_0$-measurable $\RR_+^I$-valued random variable $\ch W'(0)$ with distribution $\pi$, and let
	$\{\ch W'(t)\}$ be the stationary process defined as
	\begin{equation}\label{eq:wstarstat}
		\ch W'(t) \doteq \Gamma(\ch W'(0) + \tilde X)(t) = \ch W'(0)+ \tilde X(t) + \ch U'(t), \qquad  t \ge 0.
		\end{equation}
		Then $\ch U' \in \tilde \cla_E$ and
		$$\tilde J^{BCP,*}_{E} \le \tilde J^{BCP}_{E}(\ch U') = \int_{\RR_+^I}\hat h(w) \pi(dw) <\infty .$$
		Suppose in addition that $\hat h$ is monotonically nondecreasing, namely if $w_1, w_2 \in \RR_+^I$ satisfy $w_1 \le w_2$ then
		$\hat h(w_1) \le \hat h(w_2)$.	
	Then 
	$$\tilde J^{BCP,*}_{E} = \tilde J^{BCP}_{E}(\ch U') = \int_{\RR_+^I}\hat h(w) \pi(dw).$$
\end{theorem}

\section{Main result}
\label{sec:Main}

We now present the main result of this work.

\begin{theorem}\label{thm:main} 
The optimal cost under  the Brownian control problem provides a lower bound for the optimal asymptotic cost in the resource sharing network. Namely, 
 $J^*_E \ge \tilde J^{BCP,*}_E$
\end{theorem}

\begin{remark} \upshape
	\label{rem:rem1}
	In \cite{budjoh2}, under the conditions assumed in the current work, the condition that $\theta<0$, and an additional exponential integrability assumption on the interrarrival times and job-sizes (see Condition 1 in \cite{budjoh2}), explicit threshold form admissible control policies $\{\ch {B}^{r}\}$ are constructed for which the following holds:
 If $\sup_r \hat{q}^r<\infty$ then $J_E^r(\ch B^{r}) \rightarrow \tilde J_E^{BCP}(\ch U')$. 
Here, $\ch U'$ is the reflected Brownian motion control defined in Theorem \ref{thm:monotsoln} for the ergodic cost Brownian control problem.  In view of Theorems \ref{thm:monotsoln} and \ref{thm:main}, we then have that under the conditions of \cite{budjoh2} and with the additional assumption that
	$\hat{h}$ is nondecreasing, the sequence of control policies $\{\ch B^{r}\}$ of \cite{budjoh2} is asymptotically optimal for the ergodic cost problem.  Namely, we have that
	$$ \tilde J^{BCP,*}_E \le J^*_E \le \liminf_{r\to \infty} J_E^r(\ch B^{r})\le \limsup_{r\to \infty} J_E^r(\ch B^{r}) = \tilde J^{BCP,*}_E,$$
equivalently, under these assumptions, $J^{*}_E = \tilde J^{BCP,*}_E$.
\end{remark}

The rest of this work is devoted to the proof of Theorem \ref{thm:main}. First, in Section \ref{sec:prelim}, we introduce some preliminary results and notation needed in the proof of the theorem.  This is divided  into  two subsections; Subsection \ref{sec:detRenewProc} which is primarily focused on the underlying renewal processes and Subsection \ref{sec:measSpaceDet} which focuses on the topology on the space $\sM^I$ that will be used when proving the tightness and convergences of the controls.  
In Section \ref{sec:erg} we prove our main result. 
Finally, Section \ref{sec:MovedProofs} contains the proofs of some supplementary results that were postponed to streamline the exposition.

\section{Preliminary results}
\label{sec:prelim}

In this section we give some preparatory  results.

\subsection{The renewal processes and cost function}
\label{sec:detRenewProc}

We first introduce some additional notation that will be used in the remainder of the paper.  In the proof of Theorem \ref{thm:main} we will consider increments of $\hat{X}^r$ after some time $t> 0$, i.e., $\hat{X}^r(t+\cdot)-\hat{X}^r(t)$, and we will be concerned with the extent to which these increments depend on $\sG^r(t)$.  To this end we introduce the following.
\begin{definition}
\label{def:xiAndUpsilon}
For  $t\in [0,\infty)$, $r\in \NN$, and $j\in \AAA_J$ define
\begin{align*}
\bar{\xi}^{A,r}_{j}(t)\doteq\frac{1}{r^{2}}\sum_{l=1}^{\tau^{r,A}_{j}(t)}u^{r}_{j}(l) \qquad \mbox{and} \qquad
\bar{\xi}^{S,r}_{j}(t)\doteq\frac{1}{r^{2}}\sum_{l=1}^{\tau^{r,S}_{j}(t)}v^{r}_{j}(l).
\end{align*}
Moreover, let 
\begin{equation*}
\bar{\Upsilon}^{A,r}_{j}(t)\doteq 
\bar{\xi}^{A,r}_{j}(t)-t \qquad\mbox{and}\qquad
\bar{\Upsilon}^{S,r}_{j}(t)\doteq 
\bar{\xi}^{S,r}_{j}(t)-\bar{B}_{j}^{r}(t),
\end{equation*}
and lastly define 
\begin{equation*}
 \hat{\Upsilon}^{A,r}_{j}(t)=r\bar{\Upsilon}^{A,r}_{j}(t) \qquad\mbox{and}\qquad 
\hat{\Upsilon}^{S,r}_{j}(t)=r\bar{\Upsilon}^{S,r}_{j}(t). 
\end{equation*}
\end{definition}
Observe that all the random variables in Definition \ref{def:xiAndUpsilon} are $\sG^r(t)$-measurable.
\begin{definition}
\label{def:newStartArrAndServ}
For $t \in \R_+$ and $j \in \AAA_J$, define
\begin{equation*}  
A^{r,t}_{j}(s)\doteq\max\left\{n\geq 0:\sum_{l=\tau^{r,A}_{j}(t)+1}^{\tau^{r,A}_{j}(t)+n}u^{r}_{j}(l)\leq s\right\},
\end{equation*}
and 
\begin{equation*}
S^{r,t}_{j}(s)\doteq\max\left\{n\geq 0:\sum_{l=\tau^{r,S}_{j}(t)+1}^{\tau^{r,S}_{j}(t)+n}v^{r}_{j}(l)\leq s\right\},
\end{equation*}
along with their diffusion-scaled versions
\begin{equation*} 
\hat{A}^{r,t}_{j}(s)\doteq\frac{1}{r}A^{r,t}_{j}(r^2s)
-rs\alpha^{r}_{j} \qquad \mbox{and}\qquad 
\hat{S}^{r,t}_{j}(s)\doteq\frac{1}{r}S^{r,t}_{j}(r^2s)-rs\beta^{r}_{j}.
\end{equation*} 
\end{definition}
The following lemma describes useful properties of $(\hat{A}^{r,t},\hat{S}^{r,t})$. The proof is standard using properties of stopping times and is therefore omitted.
\begin{lemma}
\label{lem:sameDist}
Let $t\in \RR_+$  and $r \in \NN$.
Then $(\hat{A}^{r,t},\hat{S}^{r,t})$ is independent of $\mathcal{G}^{r}(t)$, and $(\hat{A}^{r,t},\hat{S}^{r,t})$ has the same distribution as $(\hat{A}^{r},\hat{S}^{r})$.
\end{lemma}

Note that we may write (cf. \cite[Equation 9]{budjoh2}), for all $j\in\AAA_J$ and $0\leq s<t$, 
\begin{align}
\hat{Q}^{r}_{j}(t)&=\hat{Q}^{r}_{j}(s)+r^{-1}\mathcal{I}_{\{t-s\geq \bar{\Upsilon}^{A,r}_{j}(s)>0\}}+r^{-1}A_{j}^{r,s}\left(r^2\left(t-s-\bar{\Upsilon}^{A,r}_{j}(s)\right)^{+}\right)\nonumber\\
&\quad -r^{-1}\mathcal{I}_{\{\bar{B}^{r}_{j}(t)- \bar{B}^{r}_{j}(s)\geq \bar{\Upsilon}^{S,r}_{j}(s)>0\}}-r^{-1}S_{j}^{r,s}\left(r^2\left(\bar{B}^{r}_{j}(t)-\bar{B}^{r}_{j}(s)- \bar{\Upsilon}^{S,r}_{j}(s)\right)^{+}\right).\label{eq:328}
\end{align}
In a similar fashion, for $0\le s\le t$,
\begin{equation}\label{eq:XhatDiff}
\begin{aligned}
\hat{X}^{r}(t)&=\hat{X}^{r}(s)+ KM^{r}\hat{A}^{r,s}\left(\left(t-s-\bar{\Upsilon}^{A,r}(s)\right)^{+}\right)-KM^{r}\hat{S}^{r,s}\left(\left(\bar{B}^{r}(t)-\bar{B}^{r}(s)-\bar{\Upsilon}^{S,r}(s)\right)^{+}\right) \\
&\quad+\frac{1}{r}KM^{r}\mathcal{I}_{\{t-s\geq\bar{\Upsilon}^{A,r}(s)>0\}}-\frac{1}{r}KM^{r}\mathcal{I}_{\{\bar{B}^{r}(t)-\bar{B}^{r}(s)\geq\bar{\Upsilon}^{S,r}(s)>0\}} \\
&\quad+r(t-s)K(\rho^{r}-\rho) -rK\rho^{r}\left((t-s)\wedge \bar{\Upsilon}^{A,r}(s)\right) +rK\left((\bar{B}^{r}(t)-\bar{B}^{r}(s))\wedge \bar{\Upsilon}^{S,r}(s)\right).
\end{aligned} 
\end{equation}

The proof of the following lemma provides some useful properties of the holding cost and the corresponding effective cost function on the workload.  Its proof is omitted because it follows easily from the local traffic condition (Condition \ref{cond:locTraffic}) and the fact that $h_j>0$ for all $j\in\AAA_J$.
\begin{lemma}
\label{lem:hHatBnds}
There exists a constant $c_h\in (1,\infty)$ such that for all $w\in\mathbb{R}_+^I$, $q\in \mathbb{R}^J_+$, and $r\in\mathbb{N}$ we have 
\begin{equation*}
c_h|w|\geq  \hat{h}(w)\geq c_h^{-1}|w|
\end{equation*}
and
\begin{equation*}
c_h|KM^r q|\geq h\cdot q \geq c_h^{-1} |KM^r q|
\end{equation*}
\end{lemma}
The next result follows from the functional central limit theorem for renewal processes.
\begin{lemma}
\label{CLTLLN}
Recall the definitions of $\theta$, $\Sigma^u$, $\Sigma^v$, and $\Sigma$ in \eqref{eq:Sigma}.
 \begin{enumerate}[(i)] \item   The following central limit theorem holds: 
\[
	(\hat{A}^r, \hat{S}^r, KM^r(\hat{A}^r -\hat{S}^r(\rho \iota)) +rK(\rho^r-\rho)\iota) \Rightarrow (\hat{A}, \hat{S}, \hat{X}) \qquad\text{in $\cld^{2J+I}$, as $r\to\infty$},
\]
where $\hat{A}$ and $\hat{S}$ are independent $J$-dimensional Brownian motions with drift $0$ and covariances $\Sigma^u$ and $\Sigma^v$, respectively, and $\hat{X}=KM(\hat{A}+\hat{S}(\rho\iota))+ \theta \iota$. In particular, $\hat{X}$ is an $I$-dimensional Brownian motion with drift $\theta$ and covariance $\Sigma$.

\item The following law of large numbers holds: 
\[
	(\bar{A}^r, \bar{S}^r) \overset{P}{\to} ( \alpha \iota, \beta \iota) \qquad\text{in $\cld^{2J}$ as $r\to \infty$}.
\]

\end{enumerate}
\end{lemma}

\begin{proof} The first statement is just the functional central limit theorem for renewal processes (see Theorem 14.6 in \cite{BillingsleyConv}), and the independence of $A$ and $S$ follows from the independence of $\{u^r_j(k)\}$ and $\{v^r_j(k)\}$.  The second statement follows from the first statement, on observing that 
\begin{equation}\label{alphabetastar}
	(\bar{A}^r, \bar{S}^r) = \left(\frac{\hat{A}^r}{r} + \alpha^r\iota, \frac{\hat{S}^r}{r} + \beta^r\iota\right) \overset{P}{\to} (\alpha \iota, \beta \iota). 
\end{equation}
\end{proof}
The following lemma follows from standard arguments. See e.g. \cite[Lemma 3.5]{budgho2}.  Proof is omitted.
\begin{lemma}\label{lem:SandABnd} There is a constant $c\in (0,\infty)$ such that for all $j \in \AAA_J$, $r\in\mathbb{N}$, and $t\ge 0$, 
\[
	E\left[ \sup_{0\le s\le t} \hat{A}^r_j(s)^2 \right] + E\left[ \sup_{0\le s\le t} \hat{S}^r_j(s)^2 \right] \le c(t+1).
\]
\end{lemma}

The proof of the following Proposition is very similar to the proof of Proposition 7.5 in \cite{budjoh2}, except here we are not assuming that the moment generating functions of the interarrival times are uniformly bounded in a neighborhood of the origin (see Condition 1 in \cite{budjoh2}). Thus, instead of uniform exponential integrability, we get a weaker result. See Section \ref{sec:proofNextArrTimeBnd} for the proof. We take the convention that $u^r_j(0)=0$ for all $r\in\mathbb{N}$ and $j\in\AAA_J$.
\begin{proposition}
\label{prop:nextArrTimeBnd}
There exists $R\in(0,\infty)$ such that for all  $j\in\AAA_J$ we have 
\begin{equation*}
\sup_{t\in [0,\infty )} \sup_{r\ge R} E\left[u^{r}_{j}(\tau^{A,r}_{j}(t))\right] <\infty.
\end{equation*}
\end{proposition}

%

\subsection{The space $\sM^I$}
\label{sec:measSpaceDet}
Recall that $\sM^d$ is the $d$-fold product space of nonnegative, locally finite measures on $[0,\infty)$ and $\sN_+^d$ is the  subset of $\sD_+^d$ consisting of nondecreasing functions.  We will use the connection between $\sN_+^I$ and $\sM^I$ to prove tightness and convergence of our controls.  
Note that for any $d\in \mathbb{N}$ there is a one-to-one correspondence between $\sN_+^d$ and $\sM^d$ where for any $g\in \sN_+^d$ the corresponding $\mu^g\in\sM^d$ is defined by
\begin{equation*}
\mu_i^g([0,t])\doteq g_i(t)\,\text{ for all }\,t\in [0,\infty)\,\text{ and }\,i\in\AAA_d, 
\end{equation*}
and for any $\mu\in \sM^d$ the corresponding $g^\mu \in \sN_+^d$ is defined by 
\begin{equation*}
g^\mu_i(t)\doteq \mu_i([0,t])\,\text{ for all }\,t\in [0,\infty)\,\text{ and }\,i\in\AAA_d.
\end{equation*}
Note that with $g \in \sN_+^d$ and $\mu \in \sM^d$ identified in the above fashion,
$g(0) = \mu(\{0\})$.

Recall that we endow $\sM^d$ with the weakest topology such that, for any $f\in \mathcal{C}_0([0,\infty))$ and $i\in\AAA_d$, the function $\mu\mapsto \int_0^\infty f(x)\mu_i(dx)$ from $\sM^d$ to $\mathbb{R}$ is continuous.  We use the distance defined below to make this a Polish space.  This metric is  a straightforward extension to multiple dimensions of the metric described in \cite[Appendix A.4]{buddup2019}.
\begin{definition}
For all $k\in\mathbb{N}$ define
\begin{equation*}
\mathcal{BL}^0_k\doteq \left\{f\in \mathcal{C}([0,k]): \sup_{t\in [0,k]}|f(t)|\leq 1\text{, }\sup_{t,s\in [0,k]}\frac{|f(t)-f(s)|}{|t-s|}\leq 1\text{, and }f(k)=0\right\}
\end{equation*} 
and 
\begin{equation*}
d_{\sM,k}(\mu,\nu)\doteq \sup_{f\in \mathcal{BL}^0_k}\left|\int_{[0,k]}f(t)\mu(dt)-\int_{[0,k]}f(t)\nu(dt)\right|, \qquad \mu,\nu\in\sM^1. 
\end{equation*}
Then the distance on $\sM^1$ given by 
\begin{equation*}
d_{\sM}(\mu,\nu)\doteq \sum_{k=1}^\infty 2^{-k}\left(d_{\sM,k}(\mu,\nu)\wedge 1\right), \qquad \mu,\nu\in\sM^1, 
\end{equation*}
and the corresponding distance  on $\sM^d$ is
\begin{equation*}
d_{\sM^d}(\mu,\nu)\doteq \sum_{i=1}^d d_{\sM}(\mu_i,\nu_i), \qquad \mu,\nu\in\sM^d. 
\end{equation*} 
\end{definition}
It is easily seen that $\lim_{n\rightarrow\infty} d_{\sM^d}(\mu^n,\mu)=0$ if and only if $\mu_i^n\rightarrow \mu_i$ in the weak topology as finite measures on $[0,k)$ for all $i\in \AAA_d$ and $k\in\mathbb{N}$.  It can be verified that $\sM^d$ is a Polish space under the distance $d_{\sM^d}$.  Although we often identify elements of $\sM^d$ and $\sN_+^d$, when referring to the space $\sM^d$ it is implied that we are using the topology from the $d_{\sM^d}$ distance and when referring to the space $\sN_+^d$ it is implied that we are using the stronger Skorokhod topology.  

Recall that $\mu^n([0,\cdot])\rightarrow\mu([0,\cdot]))$ in $\sN^d_+$ implies $\mu^n_i([0,t])\rightarrow \mu_i([0,t])$ at any continuity point $t$ of $\mu_i([0,\cdot])$, $i \in \NN_d$ (cf.  \cite[Theorem 12.5 (i)]{BillingsleyConv}).  This implies that $\mu^n\rightarrow \mu$  in $\sM^d$. So, convergence in $\sN_+^d$ implies convergence in $\sM^d$.  However, it is easily checked that  the opposite is not true.  
Consequently, if we identify the elements of $\sN^d_+$ and $\sM^d$ as discussed above, the Skorokhod topology is finer than the vague convergence topology.  We choose to use the vague convergence topology for our controls because it is sufficient for our purposes and proving tightness in this topology is easier.  
The key fact that allows us to use the tightness in the coarser vague convergence topology is that the Borel sigma fields  corresponding to the vague convergence topology and the Skorohod topology (relative to $\sN^d_+$ ) are the same.
This result is given in Lemma \ref{Borel} below. 
We start with the following definition.
\begin{definition}
\label{def:piProj}
For $t\in [0,\infty)$ define the projection $\pi_t$ from $\sD^d$ to $\mathbb{R}^d$ by $\pi_t(x)=x(t)$ for $x\in \sD^d$.  With an abuse of notation, we denote the analogous map on $\sN_+^d$ by $\pi_t$ as well. Namely,   $\pi_t: \sM^d \to  \mathbb{R}^d$ is given by $\pi_t(\mu)=\mu([0,t])$ for  $\mu\in \sM^d$.  
\end{definition}
It is known that  (cf. \cite[Proposition 3.7.1]{ethier2009markov}) that $\mathcal{B}(\sD^d) = \sigma\{\pi_t: t\ge 0\}$.
  Thus, since  $\sN_+^d$ is equipped with the topology inherited from $\sD^d$, it follows that, restricted to $\sN_+^d$, $\mathcal{B}(\sN_+^d) = \sigma\{\pi_t: t\ge 0\}$ as well.  The following lemma says that 
  $\sigma\{\pi_t: t\ge 0\}$ also equals $\clb(\sM^d)$ and consequently, 
   although the  Skorokhod topology is finer than  the topology of vague convergence, they both generate the same Borel $\sigma$-algebras. 
  The proof of the lemma follows on noting first that, by a simple monotone class argument,
  $\sigma\{\pi_t: t \in [0, \infty)\}$ is the same as $\sigma\{ \pi_A: A \in \clb([0, \infty))\}$
  where $\pi_A: \clm^1 \to [0,\infty]$ is defined as $\pi_A(\nu) \doteq \nu(A)$; and then by a standard approximation argument that the latter sigma field is the same as
  $\sigma\{ \pi_f: f \in \clc_0[0, \infty)\}$ where $\pi_f: \clm^1 \to [0,\infty)$ is defined as $\pi_f(\nu) \doteq \int f d\nu$. We omit the details.
\begin{lemma} \label{Borel}
We have,  $\mathcal{B}(\sM^d)=\sigma\left(\pi_t: t\in [0,\infty)\right)$.  In particular, the map from 
 $\sM^d$ to  $\sN_+^d$ given by $\mu\rightarrow \mu([0,\cdot])$ is measurable. 
\end{lemma}


The following lemma allows us to easily prove tightness of controls when working with the vague convergence topology and is our main motivation for using the space $\sM^I$.  The proof is straightfoward and therefore omitted. \begin{lemma}
\label{lem:compactMeasSpace}
For any sequence $K=(K_1,K_2,...)$ satisfying $K_n\in [0,\infty)$ for all $n\in\mathbb{N}$ the set 
\begin{equation*}
C_K\doteq\{\mu\in \sM^d: |\mu([0,n))|\leq K_n\text{ for all }n\in\mathbb{N}\}
\end{equation*}
is a relatively compact subset of $\sM^d$. Furthermore, if $\{\theta_k, k \in \NN\}$ is a sequence of $\clm^d$ valued random variables, then this sequence is tight (namely the corresponding probability laws are relative compact), if and only if the sequence 
$\{\theta_k[0,n), k \in \N\}$ is tight for each $n \in \NN$.
\end{lemma}

The next lemma gives an important continuity property that makes use of the above identification between $\cln_+^d$ and $\clm^d$.

\begin{lemma}
\label{lem:intBndContFnc}
Suppose that $f: \mathbb{R}^d\times\mathbb{R}_+ \to \R$ is measurable and satisfies
\begin{enumerate}[(i)]

		\item $f(\cdot,t)$ is a continuous for a.e. $t\in \mathbb{R}_+$, and 

		\item there exists a measurable function $K:\mathbb{R}_+\rightarrow\mathbb{R}_+$ such that $\int_0^\infty K(t)dt<\infty$ and for all $t\in\mathbb{R}_+$ we have $\sup_{y\in\mathbb{R}^d}|f(y,t)|\leq K(t)$. 
\end{enumerate}
Then, the real-valued function on $\mathbb{R}^d\times \sD^d\times \sM^d$ given by
\begin{equation*}
(w,x,\upsilon) \mapsto \int_0^\infty f(w+x(t)+\upsilon([0,t]),t)dt
\end{equation*}
is bounded and continuous.
\end{lemma}
The above lemma follows on making the following observations. If $\{(w^n,x^n,\upsilon^n)\}_n$ is a  sequence in $\mathbb{R}^d\times\sD^d\times \sM^d$ such $\lim_{n\rightarrow\infty}(w^n,x^n,\upsilon^n)=(w,x,\upsilon)$, then for a.e. $t$, $(w^n, x^n(t), \upsilon^n([0,t])) \to (w, x(t), \upsilon([0,t]))$ in $\R^{3d}$. The 
convergence of 
$\int_0^\infty f(w_n+x_n(t)+\upsilon_n([0,t]),t)dt$ to $\int_0^\infty f(w+x(t)+\upsilon([0,t]),t)dt$ is then an immediate consequence of dominated convergence theorem which proves the continuity of the above map. The boundedness of the map is clear.  We omit the details.

\section{Proof of Theorem \ref{thm:main}}
\label{sec:erg}

Here we prove our main result. Recall the definition of $J_E^*$ given in \eqref{eq:asymValueFcn}. 
Note that if $J_E^*=\infty$, then Theorem \ref{thm:main}  holds trivially. Thus throughout the rest of this section we assume  that $J_E^*<\infty$.  We start by considering a subsequence of networks and corresponding controls which achieve the asymptotically optimal cost.
\begin{definition}
\label{def:optSeq}
Assume $J_E^*<\infty$.  Let $\{r_{k},B^{r_{k}},T_{k}\}_{k\in\mathbb{N}}$ be a sequence such that $r_k\in\mathbb{N}$, $B^{r_k}\in\mathcal{A}^{r_k}$, and $T_k\in\mathbb{R}_+$ for all $k\in\mathbb{N}$, and which satisfies 
\begin{equation*}
r_{k} \uparrow \infty,\;\; T_{k} \uparrow \infty, \,\text{ and }\,\sup_{T\geq T_{k}}J_{E}^{r_{k},T}(B^{r_{k}})\leq J_{E}^{*}+\frac{1}{r_{k}}, 
\end{equation*}
where $J_E^{r,T}(B^r) = E\left[T^{-1}\int_0^T h \cdot \hat Q^r(t)dt\right]$. 
\end{definition}

Note that one can always find a sequence with properties as in the above definition. For notational simplicity, we will frequently write the superscripts $r_k$ in various quantities (e.g. as in $B^{r_k}$) as simply $k$ (i.e. as $B^k$).

The next two propositions are central to the proof of Theorem \ref{thm:main}, but their proofs are somewhat technical and given in Section \ref{sec:MovedProofs}.  The first allows us to control the impact of the residual arrival and service times.  In this proposition, in contrast to
$r_k\bar{\Upsilon}^{A,k}_{j}(t)$,
we only have a `time-averaged' control over $r_k\bar{\Upsilon}^{S,k}_{j}(t)$.
which is obtained by leveraging the bound on $\sup_{k}J_E^{r_k}(B^{r_k})$ in Definition 
\ref{def:optSeq}.

\begin{proposition}
\label{prop:ArrServBnds}
Consider the sequence $\{r_{k},B^{k},T_{k}\}_{k\in\mathbb{N}}$ from Definition \ref{def:optSeq}.  For any $j\in \AAA_J$ and $\epsilon>0$, we have
\begin{enumerate}[(i)]

\item $\displaystyle \lim_{k\rightarrow\infty}\sup_{t\geq 0}P\left( \hat{\Upsilon}^{A,k}_{j}(t)>\epsilon\right)=0$, and 

\item $\displaystyle \lim_{k\rightarrow\infty}\frac{1}{T_k}\int_0^{T_k}P\left( \hat{\Upsilon}^{S,k}_{j}(t)>\epsilon\right)dt=0$. 

\end{enumerate}
\end{proposition}

The first part of the next proposition follows in a straightforward manner from
Proposition \ref{prop:ArrServBnds}(i).
The main work is in proving the second part which says that the (time-averaged) allocation policy $\bar{B}^r$ is suitably close to $\rho$.  
\begin{proposition}
\label{prop:FuildScaledAllocConv}
Consider the sequence $\{r_{k},B^{k},T_{k}\}_{k\in\mathbb{N}}$ from Definition \ref{def:optSeq}.  For any $u\in[0,\infty)$ and $\epsilon>0$, we have
\begin{enumerate}[(i)]

\item $\displaystyle \lim_{k\rightarrow\infty}\sup_{t\geq 0}P\left(\sup_{s\in [0,u]}\left|\left(s-\bar{\Upsilon}^{A,k}(t)\right)^+-s\right|>\epsilon\right)=0$, and 

\item $\displaystyle \lim_{k\rightarrow\infty}\frac{1}{T_k}\int_0^{T_k}P\left(\sup_{s\in [0,u]}\Big|\left(\bar{B}^{k}(s+t)-\bar{\xi}^{S,k}(t)\right)^+-\rho s\Big|>\epsilon\right)dt=0$. 
%
%
\end{enumerate}
\end{proposition}

Recall the process $\hat U^r$ associated with a control $B^r$, as defined in Section 
\ref{sec:stateProc}. We now introduce the path occupation measures that are key to our analysis.
\begin{definition}
\label{def:occMeas}
Given $r\in\mathbb{N}$, $T\in \mathbb{R}_+$, and a control $B^r\in \mathcal{A}^r$ define the corresponding $\sM^I$-valued stochastic process $\hat{\upsilon}^r = \{\hat{\upsilon}_t^r : t \ge 0\}$ by $\hat\upsilon_t^r(\{0\}) = 0$ and 
\[
	\hat{\upsilon}^r_t( [0,s]) \doteq \hat U^r(t+s) - \hat U^r(t), \qquad s>0. 
\]
Define the random variables $\{\mu^{T,r}\}$ in $\sP(\mathbb{R}_{+}^{I}\times \sD^I\times \sM^I)$ by 
\begin{equation*}
\mu^{T,r} \doteq \frac{1}{T}\int_{0}^{T}\delta_{\left(\hat{W}^{r}(t),\hat{X}^{r}(t+\cdot)-\hat{X}^{r}(t),\hat{\upsilon}^{r}_t\right)}\,dt. 
\end{equation*}
\end{definition}
Since, $|\hat{U}_i^r(t+s)-\hat{U}_i^r(t)|\leq rC_i |t-s|$ for all $r\in\mathbb{N}$, $i\in\AAA_I$, and $t,s\in\mathbb{R}_+$, $\hat{\upsilon}^r_t$ is indeed a $\clm^I$ valued process.
Note that $\mu^{T,r}$ depends on the control $B^r$ but that this dependence is left out of the notation for the sake of brevity. Also, frequently we write $\mu^{T_k,r_k}$ as simply $\mu^k$.
In what follows,we will several times make use of the fact that a sequence $\{\gamma_k, k \in \NN\}$ of $\clp(\cls)$ valued random variables (where $\cls$ is some Polish space) is tight (namely the corresponding probability laws are relatively compact in $\clp(\clp(S))$), if and only if the sequence
$\{E(\gamma_k), k \in \NN\}$ is relatively compact in $\clp(\cls)$. (cf. \cite[Theorem 2.11]{buddup2019}).

\begin{theorem}
\label{thm:tightness}
Let $\{r_{k},B^{k},T_{k}\}_{k\in\mathbb{N}}$ be the sequence in Definition \ref{def:optSeq}, and let $\{\mu^k\}_{k\in\mathbb{N}}$ be the corresponding sequence of random measures given by Definition \ref{def:occMeas}.  Then $\{\mu^k\}_{k\in\mathbb{N}}$ is tight in $\sP(\mathbb{R}_{+}^{I}\times \sD^I\times \sM^I)$.
\end{theorem}

\begin{proof}
We will prove the result by showing that each of the three marginals are tight.  We will start with the first marginals, $\{\mu^k_{(1)}\}_{k=1}^\infty$.

Due to Lemma \ref{lem:hHatBnds} there exists $c_h<\infty$ such that for all $w\in \mathbb{R}^I_+$ we have $|w|\leq c_h \hat{h}(w)$, which gives
\begin{align*}
&E\left[\int_{\mathbb{R}_{+}^{I}}|w|\mu_{(1)}^k(dw)\right]=E\left[\frac{1}{T_{k}}\int_{0}^{T_{k}}|\hat{W}^{k}(s)|ds\right]\leq E\left[\frac{c_h}{T_{k}}\int_{0}^{T_{k}} h\cdot \hat{Q}^{k}(s)ds\right]\leq  c_h\left(J_{E}^{*}+\frac{1}{r_{k}}\right),
\end{align*}
where the last inequality uses the properties of $(r_k,B^{k},T_k)$ given in Definition \ref{def:optSeq}.  The tightness of the first marginals follows. 

 We will now show that the second marginals, $\{\mu^k_{(2)}(dx)\}_{k=1}^\infty$, are tight.
 It suffices to show that for any $\epsilon_1,\epsilon_2>0$ and $T \in [0,\infty)$, there exists $\delta>0$ and $K \in[0,\infty)$ such that for all $k\geq K$ we have
\begin{equation}
\label{eq:mu2ndMargTight}
E\left[\mu_{(2)}^k\left(\sup_{s,t\in [0,T], |t-s|< \delta}|x(t)-x(s)| >\epsilon_1 \right) \right]<\epsilon_2, 
\end{equation}
where $x$ denotes the coordinate variable on $\mathcal{D}^I$  (note that this actually gives the stronger $\sC$-tightness result).
Observe that the left side of equation \eqref{eq:mu2ndMargTight} equals 
\begin{equation*}
\frac{1}{T_k}\int_0^{T_k}P\left(\sup_{s,t\in [u,u+T], |t-s|< \delta}\left|\hat{X}^{k}(t)-\hat{X}^{k}(s) \right| >\epsilon_1\right)du.
\end{equation*}
Recall \eqref{eq:XhatDiff} and that for all $s,t\in [0,\infty)$, $j\in \AAA_J$, and $r\in\mathbb{N}$ we have $|\bar{B}^r(t)-\bar{B}^r(s)|\leq L|t-s|$,
where  $L = \max_i\{C_i\}$ 

Consequently, to prove \eqref{eq:mu2ndMargTight} it is sufficient to show that for any $\epsilon_1,\epsilon_2>0$ and $T\in[0,\infty)$, there exists $\delta>0$ and $K\in[0,\infty)$ such that for all $k\geq K$ we have the following: 
\begin{enumerate}[(a)]

\item $\displaystyle \frac{1}{T_k}\int_0^{T_k}P\left(\sup_{s,t\in [0,T], |t-s|< \delta}\left|\hat{A}^{k,u}(t)-\hat{A}^{k,u}(s) \right| >\epsilon_1\right)du<\epsilon_2$, 

\item $\displaystyle \frac{1}{T_k}\int_0^{T_k}P\left(\sup_{s,t\in [0,LT], |t-s|< L\delta}\left|\hat{S}^{k,u}(t)-\hat{S}^{k,u}(s) \right| >\epsilon_1\right)du<\epsilon_2$,

\item $\displaystyle \frac{1}{T_k}\int_0^{T_k}P\left(r_k |\bar{\Upsilon}^{A,k}(u)|+r_k |\bar{\Upsilon}^{S,k}(u)|>\epsilon_1\right)du<\epsilon_2$, and 

\item $\displaystyle \sup_{s,t\in[0,T],|t-s|<\delta}|(t-s) r_k (\rho^{k}-\rho)|\leq \epsilon_1$. 

\end{enumerate}
%
The inequalities (a) and (b) follow from Lemma \ref{CLTLLN} (i) and Lemma \ref{lem:sameDist}, (c) follows from Proposition \ref{prop:ArrServBnds}, and (d) follows from Condition \ref{heavytraffic} and the sentence below it.  This gives \eqref{eq:mu2ndMargTight} and proves the tightness of the second marginals.

Now we show that the sequence of third marginals, $\{\mu^k_{(3)}(d\upsilon)\}_{k=1}^\infty$, is tight.    Due to Lemma \ref{lem:compactMeasSpace}, it is sufficient to show that for any $n\in \mathbb{N}$ and $\epsilon>0$, there exists corresponding $M,K\in[0,\infty)$ such that 
\begin{equation}
\label{eq:timeAveMeasTightLargeK}
\sup_{k\geq K} E\left[\mu_{(3)}^k \left( |\upsilon ([0,n])|> M \right)\right]<\epsilon, 
\end{equation}
where $\upsilon$ denotes the coordinate variable on $\mathcal{M}^I$. 
Note that
\begin{equation*}
E\left[\mu_{(3)}^k\left(|\upsilon ([0,n])|> M\right)\right]= \frac{1}{T_k}\int_0^{T_k}P\left(|\hat{U}^{k}(u+n)-\hat{U}^{k}(u)|>M\right)du.
\end{equation*}
From \eqref{eq:stateq1}, for all $u\geq 0$ and $n,k\in\mathbb{N}$, we have
\begin{equation*}
\hat{U}^{k}(u+n)-\hat{U}^{k}(u)=\hat{W}^{k}(u+n)-\hat{W}^{k}(u)-(\hat{X}^{k}(u+n)-\hat{X}^{k}(u)), 
\end{equation*}
so 
\begin{equation*}
|\hat{U}^{k}(u+n)-\hat{U}^{k}(u)| \leq |\hat{W}^{k}(u+n)|+|\hat{W}^{k}(u)|+|\hat{X}^{k}(u+n)-\hat{X}^{k}(u)|. 
\end{equation*}
Thus, tightness of third marginals will follow if we can show that there exists corresponding $M,K\in[0,\infty)$ such that
\begin{equation}
\label{eq:thirdMargTightX}
\sup_{k\geq K} \frac{1}{T_k}\int_0^{T_k}P\left(|\hat{X}^{k}(u+n)-\hat{X}^{k}(u)|>M\right)du\leq \epsilon, 
\end{equation}
and
\begin{equation}
\label{eq:thirdMargTightW}
\sup_{k\geq K} \frac{1}{T_k}\int_0^{T_k}P\left( |\hat{W}^{k}(u+n)|+|\hat{W}^{k}(u)|>M\right)du\leq \epsilon.
\end{equation}
The statement in \eqref{eq:thirdMargTightX} is a consequence of the tightness property shown in \eqref{eq:mu2ndMargTight}.
 In addition, due to Lemma \ref{lem:hHatBnds} there exists $c_h<\infty$ such that for all $w\in \mathbb{R}^I_+$ we have $c_h\hat{h}(w)\geq |w|$ which gives
\begin{align*}
\frac{1}{T_{k}}\int_{0}^{T_{k}}\left(E\left[|\hat{W}^{k}(u+n)|\right]+E\left[|\hat{W}^{k}(u)|\right]\right)du&\leq\frac{2c_h}{T_{k}}\int_{0}^{T_{k}+n}E\left[h\cdot \hat{Q}^{k}(u)\right]du\\
&\leq 2c_h\frac{T_{k}+n}{T_{k}}\left(J^{*}_{E}+\frac{1}{r_{k}}\right), 
\end{align*}
where the last line uses the properties of $(r_k,B^{k},T_k)$ given in Definition \ref{def:optSeq}.  From this \eqref{eq:thirdMargTightW} follows readily.
\end{proof}

The following continuous approximations of functions in $\sD^d$ will be useful.
\begin{definition}\label{def:psidelta}
For any $\delta>0$ and $x\in\sD^1$ define $\psi_\delta(x)\in \sC^1$ by
\begin{equation*}
\psi_\delta (x)(t)\doteq \frac{1}{\delta}\int_{t}^{t+\delta} x(u)du, \qquad t \in [0,\infty). 
\end{equation*}
For $\mu\in \sM^1$, we will write $\psi_\delta(\mu)$ to denote $\psi_\delta(\mu([0,\cdot])$.  For $x\in \sD^d$ (resp. $\mu\in\sM^d$) we define $\psi_\delta$ component-wise so $\psi_\delta(x)_i=\psi_\delta(x_i)$ (resp. $\psi_\delta(\mu)_i=\psi_\delta(\mu_i)$) for all $i\in\AAA_d$.  
\end{definition}

It is straightforward to verify that for fixed $\delta>0$, and $x\in \sD^1, \mu \in \sM^1$, the functions $\psi_\delta(x)(\cdot)$ and
 $\psi_\delta(\mu)(\cdot)$ are continuous, and due to right continuity we have $\lim_{\delta\rightarrow 0}\psi_\delta(x)(t)=x(t)$, $\lim_{\delta\rightarrow 0}\psi_\delta(\mu)(t)=\mu([0,t])$ for all $t\geq 0$.  In addition,  it is easily verified that for any $\delta>0$ and $\mu\in \sM^1$ the function $\psi_\delta(\mu)(\cdot)$ is nondecreasing. Thus, for $\mu \in \sM^d$, $\psi_\delta(\mu)\in\sN_+^d$.  The following lemma says that the map $\psi_{\delta}$ is continuous.

\begin{lemma}\label{lem:contApproxDistFun}\label{lem:contApproxSkorok} Consider sequences $\{\mu^n\}$ and $\{x^n\}$ in $\mathcal{M}^d$ and $\mathcal{D}^d$, respectively, and let $\delta > 0$ be arbitrary. 

\begin{enumerate}

\item[(i)]
If $\mu^n \rightarrow \mu$ in $\sM^d$, then $\psi_\delta(\mu^n)\rightarrow \psi_\delta(\mu)$ in $\sC^d$.

\item[(ii)] 
If $x^n\rightarrow x$ in $\sD^d$, then $\psi_\delta(x^n)\rightarrow\psi_\delta(x)$ in $\sC^d$.

\end{enumerate}
\end{lemma}
To see the proof of part (i), it is sufficient to consider $d=1$ and note that $\mu^n \rightarrow \mu$ implies $\sup_n\mu^n([0,t])  < \infty$ for all $t$ and $\mu^n([0,t]) \rightarrow \mu([0,t])$ for a.e. $t$.  This implies the pointwise convergence of $\psi_\delta(\mu^n)\rightarrow \psi_\delta(\mu)$ and that $\{\psi_\delta(\mu^n)\}_{n\ge 1}$ is equicontinuous from which uniform convergence follows.  The statement in part (ii) is standard, and follows using an argument similar to that employed for the proof of part (i) combined with the results of \cite[Problem 3.11.26]{ethier2009markov}. We omit the details.

\begin{theorem}
\label{thm:conv}
Consider a subsequence $\{\mu^{m}\}_{m=1}^{\infty}$ of the tight sequence in Theorem \ref{thm:tightness} that converges weakly to a random variable $\mu^{*}\in\sP(\mathbb{R}_{+}^{I}\times \sD^I\times \sM^I)$ defined on some probability space $(\Omega^*, \mathcal{F}^*, P^*)$. Let $(w, x, \upsilon)$ denote the coordinate variables on $\mathbb{R}_{+}^{I}\times \sD^I\times \sM^I$. Then, $P^*$-a.s., $\mu^*$ satisfies the following. 
\begin{enumerate}[(i)]
\item Under $\mu^*$, $x(\cdot)$ is  a Brownian motion with drift $\theta$ and covariance $\Sigma$ with respect to the filtration $\sF^*(t) = \sigma\{w, x(s), \upsilon([0,s]) : s\le t\}$,   
\item  if $\mathbf{w}(0) \doteq w + \upsilon(\{0\})$ and ${\bf w}(t) \doteq w + x(t) + \upsilon([0,t])$ for $t > 0$, then ${\bf w}(t) \ge 0$ for all $t\ge0$, $\mu^*$-a.s., and 
\item if $\mathbf{u}(0) \doteq 0$ and ${\bf u}(t)\doteq \upsilon((0,t])$ for $t > 0$, then 
under $\mu^*$, $({\bf w}(t + \cdot),{\bf u}(t+\cdot) - \mathbf{u}(t)) \overset{d}{=} ({\bf w}(\cdot),{\bf u}(\cdot))$ on $\cld^{2I}$, for every $t \ge 0$. That is, 
\[
	\mu^*\left(({\bf w}(t + \cdot),{\bf u}(t+\cdot) - \mathbf{u}(t)) \in \cdot\right) = \mu^*\left( ({\bf w}(\cdot),{\bf u}(\cdot)) \in \cdot \right), 
\]
for all $t$.  
\end{enumerate}
\end{theorem}

\begin{proof} 

We first prove (i).  Recall in the proof of Theorem \ref{thm:tightness} we actually proved $\sC$-tightness of the second marginals 
$\{\mu^{m}_2(\cdot)\}_{m=1}^{\infty}$ which gives $\mu^*(x\in \sC^I)=1$ $P^*$-a.s.  
Fix $k \in \NN$ and $0\le r_1\le r_2 \le \cdots \le r_k\le \tilde s <t$,
$\Phi \in \clc_b(\RR^{(1+2k)I})$, $f \in \clc_0^2(\RR^{I})$. Define
$\phi: \RR_+^I \times \cld^I \times \clm^I \to \RR$ as, for $(w,x,\upsilon) \in \RR_+^I \times \cld^I \times \clm^I$,
$$\phi(w,x,\upsilon) \doteq \Phi(w, x(r_1), \ldots x(r_k), \upsilon([0,r_1]), \ldots \upsilon([0,r_k])).$$
To prove (i) it suffices to show that
\begin{equation}
\label{eq:longTermCostBMProof}
E^* \left[\left(\int_{\R_+^I \times \sD^I \times \sM^I} \phi\left( w, x, \upsilon  \right)\left( f(x(t)) - f(x(\tilde{s})) - \int_{\tilde{s}}^{t} \sL f(x(z))\,dz \right)\, \mu^*(dw, \,dx, \,d\upsilon)\right)^2 \right] = 0,
\end{equation}
where $\sL$ is the generator of a Brownian motion with drift $\theta$ and covariance $\Sigma$. Namely, 
\[
	\sL f(x) = \sum_{i=1}^I \theta_i \frac{\partial f}{\partial x_i}(x) + \frac{1}{2}\sum_{i=1}^I\sum_{j=1}^I \Sigma_{i,j} \frac{\partial^2 f}{\partial x_i\partial x_j}(x),\;\; f \in \clc_0^2(\RR^I).
\]
One of the issues to be careful about is that we do not know that $\mu^*(\upsilon[0, \cdot) \in \clc^I)= 1$. Thus to employ a weak convergence argument, some mollification of $\upsilon$ is useful.
Fix $\delta \in (0, t-\tilde s)$ and define
$\phi^{\delta}: \RR_+^I \times \cld^I \times \clm^I \to \RR$ as, for $(w,x,\upsilon) \in \RR_+^I \times \cld^I \times \clm^I$,
$$\phi^{\delta}(w,x,\upsilon) \doteq \Phi(w, x(r_1), \ldots x(r_k), \psi_{\delta}(\upsilon)(r_1), \ldots \psi_{\delta}(\upsilon)(r_k)).$$
We will now argue that 
\begin{equation}
\label{eq:longTermCostBMProofApprox}
E^* \left[\left(\int_{\R_+^I \times \sD^I \times \sM^I} 
\phi^{\delta}\left( w, x, \upsilon  \right)
\left( f(x(t)) - f(x(\tilde s+\delta)) - \int_{\tilde s+\delta}^{t} \sL f(x(z))\,dz \right)\, \mu^*(dw, \,dx, \,d\upsilon)\right)^2 \right] = 0, 
\end{equation}
The claimed identity in \eqref{eq:longTermCostBMProof} follows immediately from this on noting that $\mu^*(x\in \sC^I)=1$ $P^*$-a.s.,
using the fact that for any $\mu \in \clm^d$, $\lim_{\delta\rightarrow 0}\psi_\delta(\mu)(t)=\mu([0,t])$ for all $t\geq 0$, and sending 
$\delta \to 0$ in the above display.


We now show (\ref{eq:longTermCostBMProofApprox}) to complete the proof of (i).  Let $s = \tilde s + \delta$.
For notational convenience, for any $u\in [0,\infty)$, $x\in \sD^I$, $\xi\in\sD^I_+$, and an $\sM^I$-valued process $\nu = \{\nu_u : u \ge 0\}$, we define 
\[
F^{u}_{s,t}(x) \doteq  f(x(u+t)-x(u)) - f(x(u+s)-x(u)) - \int_{s}^{t} \sL f(x(u+z)-x(u))\,dz, 
\]
and
\[
G^{u}_{\delta}(\xi,x,\nu) \doteq   {\phi}^{\delta}\left( \xi(u), x(u+\cdot)-x(u), \nu_u  \right).
\]
Due to the weak convergence $\mu^{m} \to \mu^*$, and recalling that $\mu^*(x\in \sC^I)=1$ $P^*$-a.s. and ${\phi}^{\delta}$ is a bounded function that is continuous for $x\in \sC^I$ due to Lemma \ref{lem:contApproxDistFun}, we have
\begin{equation}\label{eq:BMProofMeasTimeAve}
\begin{aligned}
&E^* \left[\left(\int_{\R_+^I \times \sD^I \times \sM^I} {\phi}^{\delta}\left( w, x, \upsilon  \right)\left( f(x(t)) - f(x(s)) - \int_s^{t} \sL f(x(z))\,dz \right)\, \mu^*(dw, \,dx, \,d\upsilon)\right)^2 \right] \\
&= \lim_{m\rightarrow\infty}E \left[\left(\int_{\R_+^I \times \sD^I \times \sM^I} {\phi}^{\delta}\left( w, x, \upsilon  \right)\bigg( f(x(t)) - f(x(s)) \right. \right. \\
 &\qquad\qquad\qquad\qquad\qquad\qquad\qquad\qquad\qquad \left.  \left. - \int_s^{t} \sL f(x(z))\,dz \bigg)\, \mu^{m}(dw, \,dx, \,d\upsilon)\right)^2 \right]\\
&= \lim_{m\rightarrow\infty}E \left[\left(\frac{1}{T_{m}}\int_{0}^{T_{m}}G^{u}_{\delta}( \hat{W}^{m}, \hat{X}^{m}, \hat \upsilon^{m})F^{u}_{s,t}(\hat{X}^{m})du\, \right)^2 \right], 
\end{aligned}
\end{equation}
where we recall the convention that the superscripts $r_m$ are simply written as $m$ and the process $\hat \upsilon^m = \hat\upsilon^{r_m}$ from Definition \ref{def:occMeas}. 
Note also that 
\begin{align*}
&E \left[\left(\frac{1}{T_{m}}\int_{0}^{T_{m}}G^{u}_{\delta}( \hat{W}^{m}, \hat{X}^{m}, \hat\upsilon^{m} )F^{u}_{s,t}(\hat{X}^{m})du\, \right)^2 \right]\\
&=\frac{1}{T^{2}_{m}}\int_{0}^{T_{m}} \int_{0}^{T_{m}}E\left[G^{u}_{\delta}( \hat{W}^{m}, \hat{X}^{m}, \hat\upsilon^{m})F^{u}_{s,t}
(\hat{X}^{m})G^{v}_{\delta}( \hat{W}^{m}, \hat{X}^{m}, \hat\upsilon^{m})F^{v}_{s,t}(\hat{X}^{m})\right]dvdu\\
&=\frac{2}{T^{2}_{m}}\int_{0}^{T_{m}} \int_{0}^{u}E\left[G^{u}_{\delta}( \hat{W}^{m}, \hat{X}^{m}, \hat\upsilon^{m})F^{u}_{s,t}(\hat{X}^{m})G^{v}_{\delta}( \hat{W}^{m}, \hat{X}^{m}, \hat \upsilon^{m})F^{v}_{s,t}(\hat{X}^{m})\right]dvdu. 
\end{align*}
Recall the definitions $\hat A^{r,t}$ and $\hat S^{r,t}$ from Section \ref{sec:detRenewProc} and recall that we denote
$\hat A^{r_m,t}$ and $\hat S^{r_m,t}$ as $\hat A^{m,t}$ and $\hat S^{m,t}$, respectively. For all $z,v\geq 0$ and $m\in\mathbb{N}$, define
\begin{equation*}
\tilde{X}^{m,v}(z)\doteq KM^{m}\left( \hat{A}^{m,v}(z)-\hat{S}^{m,v}(\rho z)\right) +r_{m}zK(\rho^{m}-\rho),
\end{equation*}
and 
\begin{equation*}
\breve{X}^{m,v}(z)\doteq\begin{cases}
\hat{X}^{m}(z)&\text{ if } z\leq v,\\
\hat{X}^{m}(v)+\tilde{X}^{m,v}(z-v) & \text{ if } z > v. 
\end{cases}
\end{equation*}
By Lemma \ref{CLTLLN} (i) and Lemma \ref{lem:sameDist}, for any $\epsilon>0$ there exists an $M \in [0,\infty)$ and $\kappa>0$ such that
\begin{equation*}
\sup_{u\geq 0, m\geq M}P\left( \sup_{v,z\in [0,Lt],|v-z|\leq \kappa}\Big|\hat{S}^{m,u+s}(v)- \hat{S}^{m,u+s}(z)\Big|>\epsilon\right)<\epsilon, 
\end{equation*}
and
\begin{equation*}
\sup_{u\geq 0,m\geq M}P\left( \sup_{v,z\in [0,t],|v-z|\leq \kappa}\Big|\hat{A}^{m,u+s}(v)- \hat{A}^{m,u+s}(z)\Big|>\epsilon\right)<\epsilon.
\end{equation*}
Consequently, recalling \eqref{eq:XhatDiff} and using Propositions \ref{prop:ArrServBnds} and \ref{prop:FuildScaledAllocConv}, it follows that for any $\epsilon>0$ we have
\begin{align*}
&\lim_{m\rightarrow\infty}\frac{1}{T_m}\int_0^{T_m}P\left(\sup_{z\in[0,u+t]}\Big|\hat{X}^{m}(z)-\breve{X}^{m,u+s}(z) \Big|>\epsilon \right)du\\
&=\lim_{m\rightarrow\infty}\frac{1}{T_m}\int_0^{T_m}P\left(\sup_{z\in[0,t-s]}\Big|\hat{X}^{m}(u+s+z)-\hat{X}^{m}(u+s)-\tilde{X}^{m,u+s}(z) \Big|>\epsilon \right)du=0, 
\end{align*}
which, combined with the fact that $f \in \sC_0^2(\R^I)$, implies that for any $\epsilon>0$ we have
\begin{equation}
\label{eq:FXhatFXbreve}
\lim_{m\rightarrow\infty}\frac{1}{T_m}\int_0^{T_m}P\left(\Big| F^{u}_{s,t}(\hat{X}^{{m}})-F^{u}_{s,t}(\breve{X}^{{m},u+s}) \Big|>\epsilon \right)du=0.
\end{equation}
Recall from Lemma \ref{lem:sameDist} that for all $m\in \mathbb{N}$ and $v\geq 0$, we have $\tilde{X}^{m,v}\overset{d}{=}\tilde{X}^{m,0}$ and from Lemma \ref{CLTLLN} (i) that $\tilde{X}^{m,0}\Rightarrow \hat{X}$ in $\sD^I$, where $\hat{X}$ is an $I$-dimensional Brownian motion with drift $\theta$ and covariance matrix $\Sigma$. Since $f \in \sC_0^2(\R^I)$, for all $y\in\mathbb{R}^I$ we have
\begin{equation*}
E^*\left[  f(\hat{X}(t-s)+y) - f(y) - \int_{s}^{t} \sL f(\hat{X}(z-s)+y)\,dz\right]=0. 
\end{equation*}
Thus, the compact support of $f$ implies
\begin{equation}
\label{eq:limExpF}
\lim_{m\rightarrow\infty}\sup_{y\in\mathbb{R}^I,v\geq 0}\left |E\left[  f(\tilde{X}^{m,v}(t-s)+y) - f(y) - \int_{s}^{t} \sL f(\tilde{X}^{m,v}(z-s)+y)\,dz\right] \right| =0.
\end{equation}
Now, writing
\begin{align*}
F^{u}_{s,t}(\breve{X}^{m,u+s})=& f(\tilde{X}^{m,u+s}(t-s)+\hat{X}^{m}(u+s)-\hat{X}^{m}(u)) - f(\hat{X}^{m}(u+s)-\hat{X}^{m}(u)) \\
&- \int_{s}^{t} \sL f(\tilde{X}^{m,u+s}(z-s)+\hat{X}^{m}(u+s)-\hat{X}^m(u))\,dz, 
\end{align*}
and since 
$\tilde{X}^{m,u+s}(\cdot)$ is independent of $\sG^{r_m}(u+s)$ and $\hat{X}^{m}(u+s)-\hat{X}^m(u)$ is $\sG^{r_m}(u+s)$-measurable, we have
\begin{equation}\label{eq:condExpFBnd}
\begin{aligned}
&\sup_{u\geq 0}E\left[F^{u}_{s,t}(\breve{X}^{m,u+s})\Big| \sG^{m}(u+s) \right]\\
&\leq \sup_{y\in\mathbb{R}^I,u\geq 0}E\left[  f(\tilde{X}^{m,u+s}(t-s)+y) - f(y) - \int_{s}^{t} \sL f(\tilde{X}^{m,u+s}(z-s)+y)\,dz\right].
\end{aligned}
\end{equation}
Finally, 
since $G^{u}_{r,\delta}$ and $F^{u}_{s,t}$ are uniformly bounded in $u$, we have
\begin{align*}
&E^* \left[\left(\int_{\R_+^I \times \sD^I \times \sM^I} {\phi}^{\delta}\left( w, x, \upsilon  \right)\left( f(x(t)) - f(x(s)) - \int_s^{t} \sL f(x(z))\,dz \right)\, \mu^*(dw, \,dx, \,d\upsilon)\right)^2 \right] \\
&=\lim_{m\rightarrow\infty}\frac{2}{T^{2}_{m}}\int_{0}^{T_{m}} \int_{0}^{u}E\left[G^{u}_{\delta}( \hat{W}^{{m}}, \hat{X}^{{m}}, \hat\upsilon^{{m}})F^{u}_{s,t}(\hat{X}^{{m}})G^{v}_{\delta}( \hat{W}^{{m}}, \hat{X}^{{m}}, \hat\upsilon^{{m}} )F^{v}_{s,t}(\hat{X}^{{m}})\right]dvdu\\
&=\lim_{m\rightarrow\infty}\frac{2}{T^{2}_{m}}\int_{0}^{T_{m}} \int_{0}^{u-t+s}E\left[G^{u}_{\delta}( \hat{W}^{{m}}, \hat{X}^{{m}}, \hat\upsilon^{{m}})F^{u}_{s,t}(\hat{X}^{{m}})G^{v}_{\delta}( \hat{W}^{{m}}, \hat{X}^{{m}}, \hat\upsilon^{{m}})F^{v}_{s,t}(\hat{X}^{{m}})\right]dvdu\\
&=\lim_{m\rightarrow\infty}\frac{2}{T^{2}_{m}}\int_{0}^{T_{m}} \int_{0}^{u-t+s}E\left[G^{u}_{\delta}( \hat{W}^{{m}}, \hat{X}^{{m}}, \hat\upsilon^{{m}})F^{u}_{s,t}(\breve{X}^{{m},u+s}) \right. \\
&\hspace{2in} \times \left.  G^{v}_{\delta}( \hat{W}^{{m}}, \hat{X}^{{m}}, \hat\upsilon^{{m}} )F^{v}_{s,t}(\hat{X}^{{m}})\right]dvdu\\
&=\lim_{m\rightarrow\infty}\frac{2}{T^{2}_{m}}\int_{0}^{T_{m}} \int_{0}^{u-t+s}E\left[E\left[F^{u}_{s,t}(\breve{X}^{{m},u+s})\Big|\sG^{r_m}(u+s) \right]G^{u}_{\delta}( \hat{W}^{{m}}, \hat{X}^{{m}}, \hat\upsilon^{{m}})\right. \\
&\hspace{2in} \times \left. G^{v}_{\delta}( \hat{W}^{{m}}, \hat{X}^{{m}}, \hat\upsilon^{{m}})F^{v}_{s,t}(\hat{X}^{{m}}) \right]dvdu \\
&= 0. 
\end{align*}
Above, the third equality uses \eqref{eq:FXhatFXbreve}, the fourth equality uses the fact that $G^{u}_{\delta}( \hat{W}^{{m}}, \hat{X}^{{m}}, \hat\upsilon^{{m}})$,  $G^{v}_{\delta}( \hat{W}^{{m}}, \hat{X}^{{m}}, \hat\upsilon^{{m}})$, and $F^{v}_{s,t}(\hat{X}^{{m}})$ are $\sG^{r_m}(u+s)$-measurable for $v\leq u-t+s$, and the last equality uses \eqref{eq:condExpFBnd} combined with \eqref{eq:limExpF}.  This shows \eqref{eq:longTermCostBMProofApprox} and completes the proof of (i).

We now prove (ii). Since $\mathbf{w}(0) = w + \upsilon(\{0\}) \ge 0$, it suffices, from the continuity and right continuity properties of $x$ and $\upsilon[0,\cdot]$, to show that,
 for all $t > 0$, 
\begin{equation}
\label{eq:posOrthTimeAve}
E^* \left[ \mu^*\left( w + x(t) + \upsilon([0,t]) \in \mathbb{R}_+^I \right) \right] = 1.
\end{equation}
Now fix $t\geq 0$.  For $\epsilon\geq 0$ define
\begin{equation*}
\mathcal{N}_\epsilon \doteq \left\{y\in\mathbb{R}^I: \inf_{ x \in\mathbb{R}^{I}_{+}}|y-x|>\epsilon\right\}, 
\end{equation*}
i.e., $\mathcal{N}_\epsilon$ is the set of values $y\in\mathbb{R}^I$ that are more than $\epsilon$ away from the positive orthant $\mathbb{R}^I_+$.  Let $\delta>0$ be arbitrary and note that due to 
Lemma \ref{lem:contApproxSkorok}, the mapping 
\begin{equation*}
(w,x,\upsilon) \mapsto w+\psi_\delta(x)(t)+\psi_\delta(\upsilon)(t), \qquad (w,x,\upsilon)\in \R_+^I \times \sD^I \times \sM^I, 
\end{equation*}
is continuous. Consequently, because $\mathcal{N}_\epsilon$ is an open set, the mapping 
\begin{equation*}
(w,x,\upsilon)\mapsto \mathcal{I}_{\mathcal{N}_\epsilon}(w+\psi_\delta(x)(t)+\psi_\delta(\upsilon)(t)), \qquad (w,x,\upsilon)\in \R_+^I \times \sD^I \times \sM^I, 
\end{equation*}
 is lower semicontinuous.  This then implies (cf. \cite[Theorem A.3.12]{Dupuis97}) that the mapping 
\begin{align*}
\mu\mapsto & \int_{\R_+^I \times \sD^I \times \sM^I} \mathcal{I}_{\mathcal{N}_\epsilon}(w+\psi_\delta (x)(t)+\psi_\delta(\upsilon)(t)) \mu(dw, \,dx, \,d\upsilon) \\
&= \mu \left( w + \psi_\delta(x)(t) + \psi_\delta(\upsilon)(t) \in \mathcal{N}_\epsilon \right), \qquad \mu\in\mathcal{P}(\R_+^I \times \sD^I \times \sM^I), 
\end{align*}
is also lower semicontinuous.  
Recalling, for $u \ge 0$, $\hat W^r(u) = \hat w^r + \hat X^r(u) + \hat U^r(u) \in \R_+^I$ a.s., note that 
\begin{equation}\label{eq:psiWXUcalc}
\begin{aligned}
	&\hat W^r(u) + \psi_\delta(\hat X^r(u + \cdot) - \hat X^r(u))(t) + \psi_\delta(\hat \upsilon_u^r)(t) \\
	&= \hat W^r(u) + \frac{1}{\delta}\int_t^{t+\delta}(\hat X^r(u+z) - \hat X^r(u))dz + \frac{1}{\delta}\int_t^{t+\delta}\hat \upsilon^r_u([0,z])dz \\
	&= \hat W^r(u) - \hat X^r(u) - \hat U^r(u) + \frac{1}{\delta} \int_t^{t+\delta} (\hat X^r(u+z) + \hat U^r(u+z))dz \\
	&= \frac{1}{\delta}\int_t^{t+\delta}\hat W^r(u+z)dz = \psi_\delta(\hat{W}^{r}(u+\cdot))(t) \in \R_+^I \;\mbox{a.s.}
\end{aligned}
\end{equation}
Then, from weak convergence of $\mu^m\to \mu^*$
\begin{align*}
E^*\left[ \mu^*\left(w+\psi_\delta (x)(t)+\psi_\delta(\upsilon)(t) \in \mathcal{N}_\epsilon\right)\right] 
&\leq\lim_{m\rightarrow\infty}E\left[ \mu^{m}\left(w+\psi_\delta (x)(t)+\psi_\delta(\upsilon)(t) \in \mathcal{N}_\epsilon \right) \right]\\
&=\lim_{m\rightarrow\infty} \frac{1}{T_m}\int_0^{T_m} P\left(\psi_\delta(\hat{W}^{m}(u+\cdot))(t) \in \mathcal{N}_\epsilon\right) du = 0.
\end{align*}
Since $\mathcal{N}_{0}=\bigcup_{n\in\mathbb{N}}\mathcal{N}_{1/n}$, it follows that
\begin{equation}\label{eq:arbdelta}
E^*\left[ \mu^*\left(w+\psi_\delta (x)(t)+\psi_\delta(\upsilon)(t) \in \mathcal{N}_0\right) \right]\leq \sum_{n\in\mathbb{N}}E^*\left[\mu^*\left(w+\psi_\delta (x)(t)+\psi_\delta(\upsilon)(t) \in \mathcal{N}_{1/n}\right)\right]=0.
\end{equation}
Next, since
\[
	w + x(t) + \upsilon([0,t]) = \lim_{\delta\to0} \left( w + \psi_{\delta}(x)(t) + \psi_{\delta}(\upsilon)(t) \right),
\]
the statement in \eqref{eq:posOrthTimeAve}  follows on sending $\delta\to 0$ in \eqref{eq:arbdelta}. This
completes the proof of (ii). 


Finally, we turn to (iii). 
For $n\in\mathbb{N}$, $\{s_{i}\}_{i=1}^n\subset\mathbb{R}_+^n$, $t\ge0$, $(y,z)\in\sD^{2I}$, and $f\in\mathcal{C}_b(\mathbb{R}^{n2I})$, 
define the quantities
\begin{equation*}
\label{eq:finDimDistContFcn}
\Psi^f_{\{s_i\}}(y,z) \doteq f\left(\pi_{\{s_i\}}(y(\cdot),z(\cdot)-z(0))\right), \qquad
\Psi^{f,t}_{\{s_i\}}(y,z) \doteq f\left(\pi_{\{s_i\}}(y(t + \cdot),z(t+\cdot)-z(t))\right),
\end{equation*}
and
\begin{equation}
\label{eq:PhiDefPsiMinusPsiT}
\Phi^{f,t}_{\{s_i\}}(y,z)\doteq \Psi^f_{\{s_i\}}(y,z)-\Psi^{f,t}_{\{s_i\}}(y,z), 
\end{equation}
where $\pi_{\{s_i\}} \doteq (\pi_{s_1},...,\pi_{s_n})$ and $\pi_t$ is the projection mapping given in Definition \ref{def:piProj}.  
Recall that (cf. \cite[Proposition 3.7.1]{ethier2009markov}), there is a countable collection of measure-determining maps on $\sD^{2I}$ of the form 
\begin{equation*}
(y,z) \mapsto f\left(\pi_{\{s_i\}}(y,z)\right), \qquad (y,z) \in \sD^{2I}. 
\end{equation*}
Together with the right continuity of $t\mapsto ({\bf w}(t + \cdot),{\bf u}(t+\cdot)-{\bf u}(t))$, this says that, to show (iii) it suffices to 
 show that for any $t\ge0$, $n\in\mathbb{N}$, $\{s_{i}\}_{i=1}^n \subset \mathbb{R}^{n}$, and $f\in\mathcal{C}_b(\mathbb{R}^{n2I})$ we have
\begin{equation}
E^*\left[\bigg|\int_{\R_+^I \times \sD^I \times \sM^I}\Phi^{f,t}_{\{s_i\}}(w + x+\upsilon([0,\cdot]),\upsilon([0,\cdot]))\mu^*(dw,dx,d\upsilon) \bigg| \right]=0. \label{eq:finDimDistContFcnTimeDiff}
\end{equation}

We now prove \eqref{eq:finDimDistContFcnTimeDiff}. Let $\delta>0$.  Due to Lemma \ref{lem:contApproxDistFun}, 
 the map
\begin{equation*}
(w,x,\upsilon) \mapsto \Phi^{f,t}_{\{s_i\}}(w+ \psi_\delta(x)+ \psi_\delta(\upsilon),\psi_\delta(\upsilon)), \qquad (w,x,\upsilon)\in\mathbb{R}_+^I\times \sD^I\times \sM^I.
\end{equation*}
is bounded and continuous. 
Then, 
\begin{equation} \label{eq:psilim0}
\begin{aligned}
&E^*\left[\bigg|\int_{\R_+^I \times \sD^I \times \sM^I}\Phi^{f,t}_{\{s_i\}}(w+ \psi_\delta(x)+ \psi_\delta(\upsilon),\psi_\delta(\upsilon))\mu^*(dw,dx,d\upsilon) \bigg| \right]\\
&=\lim_{m\rightarrow\infty}E\left[\bigg|\int_{\R_+^I \times \sD^I \times \sM^I}\Phi^{f,t}_{\{s_i\}}(w+ \psi_\delta(x)+ \psi_\delta(\upsilon),\psi_\delta(\upsilon))\mu^{m}(dw,dx,d\upsilon) \bigg| \right] \\
&=\lim_{m\rightarrow\infty}E\left[\bigg|\frac{1}{T_m}\int_0^{T_m}\Phi^{f,t}_{\{s_i\}}(\psi_\delta(\hat{W}^{m}(u+\cdot)),\psi_\delta(\hat{U}^{m}(u+\cdot)-\hat{U}^{m}(u)))du \bigg| \right], 
\end{aligned}
\end{equation}
where the second equality uses the calculation in \eqref{eq:psiWXUcalc}. 
Note that
\begin{align*}
&\Psi^f_{\{s_i\}}(\psi_\delta(\hat{W}^{m}(u+t+\cdot)),\psi_\delta(\hat{U}^{m}(u+t+\cdot)-\hat{U}^{m}(u+t)))\\
&=\Psi^{f,t}_{\{s_i\}}(\psi_\delta(\hat{W}^{m}(u+\cdot)),\psi_\delta(\hat{U}^{m}(u+\cdot)-\hat{U}^{m}(u)))
\end{align*}
for all $u,t\geq0$, and therefore
\begin{align*}
&\int_0^{T_m} \Psi^f_{\{s_i\}}(\psi_\delta(\hat{W}^{r_m}(u+\cdot)),\psi_\delta(\hat{U}^{r_m}(u+\cdot)-\hat{U}^{r_m}(u)) ) du\\
&= \int_0^t \Psi^f_{\{s_i\}}(\psi_\delta(\hat{W}^{r_m}(u+\cdot)),\psi_\delta(\hat{U}^{r_m}(u+\cdot)-\hat{U}^{r_m}(u)) ) du \\
 &\quad + \int_0^{T_m-t} \Psi^{f,t}_{\{s_i\}}(\psi_\delta(\hat{W}^{r_m}(u+\cdot)),\psi_\delta(\hat{U}^{r_m}(u+\cdot)-\hat{U}^{r_m}(u)) ) du.
 \end{align*}
 It then follows that
 \begin{align*}
 &\int_0^{T_m}\Phi^{f,t}_{\{s_i\}}(\psi_\delta(\hat{W}^{r_m}(u+\cdot)),\psi_\delta(\hat{U}^{r_m}(u+\cdot)-\hat{U}^{r_m}(u)))du  \\
 &= \int_0^t \Psi^f_{\{s_i\}}(\psi_\delta(\hat{W}^{r_m}(u+\cdot)),\psi_\delta(\hat{U}^{r_m}(u+\cdot)-\hat{U}^{r_m}(u)) ) du \\
 &\quad - \int_{T_m-t}^{T_m} \Psi^{f,t}_{\{s_i\}}(\psi_\delta(\hat{W}^{r_m}(u+\cdot)),\psi_\delta(\hat{U}^{r_m}(u+\cdot)-\hat{U}^{r_m}(u)) ) du. 
\end{align*}
Recalling the boundedness of $f$ we now see that the right side (and therefore also the left side) of \eqref{eq:psilim0} is $0$. The statement in \eqref{eq:finDimDistContFcnTimeDiff}
now follows on sending $\delta \to 0$.
This completes the proof of (iii).
\end{proof}


We now prove our main result. \\


{\bf Proof of Theorem \ref{thm:main}.} 
Note that the inequality $J_{E}^*\geq \tilde{J}_{E}^{BCP,*}$ holds trivially when $J_E^*=\infty$. So suppose now that $J_E^*<\infty$. 
Under this assumption, by  Theorems \ref{thm:tightness} and \ref{thm:conv}, there exists a subsequence $\{r_m,B^{m},T_{r_m}\}_{m=1}^{\infty}$ of the sequence defined in Definition \ref{def:optSeq} such that the corresponding random measures $\{\mu^{m}\}_{m=1}^{\infty}$ given by Definition \ref{def:occMeas} converge weakly to a random variable $\mu^*\in\mathcal{P}(\mathbb{R}_+^I\times\sD^I\times\sM^I)$,  defined on some space $(\Omega^*, \sF^*, P^*)$,  which satisfies the conclusions of Theorem \ref{thm:conv}.

For arbitrary $N\in[0,\infty)$, define the following bounded approximation of the effective cost function:
\begin{equation}
\label{eq:hHatN}
\hat{h}_{N}(w)\doteq\hat{h}(w)\wedge N, \qquad w\in\mathbb{R}_+^I. 
\end{equation}
Note that $\hat{h}_N$ is a bounded, continuous function on $\mathbb{R}_+^I$ and that $\hat{h}_N(w)\uparrow\hat{h}(w)$ as $N\rightarrow\infty$ for all $w\in\mathbb{R}_+^I$.  Let $z>0$ be arbitrary.  
 Lemma \ref{lem:intBndContFnc} applied to the function $f(y,s) = z^{-1}\hat h_N(y)\mathcal{I}_{\{s \le z\}}$ says that the map
 \[
 	(w, x, \upsilon) \mapsto \frac{1}{z}\int_0^z \hat h_N(w + x(s) + \upsilon([0,s]) ds, \qquad (w, x, \upsilon) \in \mathbb{R}_+^I\times\sD^I\times\sM^I, 
 \]
 is bounded and continuous, and so 
\begin{equation*}
\mu\mapsto\int_{\R_+^I \times \sD^I \times \sM^I}\frac{1}{z}\int_{0}^{z}\hat{h}_{N}(w+x(s)+\upsilon([0,s]))ds\mu(dw,dx,d\upsilon), \qquad \mu\in\mathcal{P}(\mathbb{R}_+^I\times\sD^I\times\sM^I), 
\end{equation*}
is a bounded and continuous function. Combined with the fact that $\mu^{m}\Rightarrow\mu^*$ in $\mathcal{P}(\mathbb{R}_+^I\times\sD^I\times\sM^I)$, this implies that 
\begin{equation}\label{eq:BndErgCostMeasConv}
\begin{aligned}
&\lim_{m\rightarrow\infty}E\left[\int_{\R_+^I \times \sD^I \times \sM^I}\frac{1}{z}\int_{0}^{z}\hat{h}_{N}(w+x(s)+\upsilon([0,s]))ds\mu^{{m}}(dw,dx,d\upsilon)\right]\\
&=E^*\left[\int_{\R_+^I \times \sD^I \times \sM^I}\frac{1}{z}\int_{0}^{z}\hat{h}_{N}(w+x(s)+\upsilon([0,s]))ds\mu^{*}(dw,dx,d\upsilon)\right].
\end{aligned}
\end{equation}
In addition, for all $m\in\mathbb{N}$ we have 
\begin{align*}
\label{eq:BndErgCostMeasUpprBnd}
&E\left[\int_{\R_+^I \times \sD^I \times \sM^I}\frac{1}{z}\int_{0}^{z}\hat{h}_{N}(w+x(s)+\upsilon([0,s]))ds\mu^{m}(dw,dx,d\mu)\right]\\
&=E\left[\frac{1}{zT_{m}}\int_{0}^{T_{m}}\int_{0}^{z}\hat{h}_{N}(\hat{W}^{{m}}(s+t))dsdt\right] 
=E\left[\frac{1}{zT_{m}}\int_{0}^{T_{m}+z}\int_{(u-z)\vee 0}^{u\wedge T_m}\hat{h}_{N}(\hat{W}^{r_{m}}(u))dtdu\right] \\
&\leq E\left[\frac{1}{T_{m}}\int_{0}^{T_{m}+z}\hat{h}_{N}(\hat{W}^{r_{m}}(u))du\right]\leq E\left[\frac{1}{T_{m}}\int_{0}^{T_{m}+z}h\cdot \hat{Q}^{r_{m}}(u)du\right]\leq \frac{T_m+z}{T_{m}}\left( J_E^{*}+\frac{1}{r_m}\right), 
\end{align*}
where the second equality uses a change of variables with $u=s+t$, the first inequality uses the fact that $\hat{h}_N(\cdot)$ is nonnegative, the second inequality uses the definitions of $\hat{h}$ and $\hat{h}_N$ in \eqref{eq:hhat}  and \eqref{eq:hHatN}, and the third inequality uses the properties of $(r_m,B^{r_m},T_m)$ given in Definition \ref{def:optSeq}.  Combining  this bound with \eqref{eq:BndErgCostMeasConv} 
implies
\begin{equation*}
E^*\left[\int_{\R_+^I \times \sD^I \times \sM^I}\frac{1}{z}\int_{0}^{z}\hat{h}_{N}(w+x(s)+\upsilon([0,s]))ds\mu^{*}(dw,dx,d\upsilon)\right]\leq J_E^*
\end{equation*}
Since this holds for all $z>0$, and right continuity implies
\begin{equation*}
\lim_{z\rightarrow 0}\frac{1}{z}\int_{0}^{z}\hat{h}_N(w+x(s)+\upsilon([0,s]))ds = \hat{h}_N(w+x(0)+\upsilon(\{0\})), 
\end{equation*}
for all $(w,x,\upsilon)\in \mathbb{R}_+^I\times \sD^I\times\sM^I$, the dominated convergence theorem (combined with the fact that $\mu^*(\{x(0)=0\})=1$ $P^*$-a.s.) gives
\begin{equation}\label{eq:BndErgCostMeasBndUnifN}
\begin{aligned}
&E^*\left[\int_{\R_+^I \times \sD^I \times \sM^I}\hat{h}_N(w+\upsilon(\{0\}))\mu^{*}(dw,dx,d\upsilon)\right]\\
&=\lim_{z\rightarrow 0}E\left[\int_{\R_+^I \times \sD^I \times \sM^I}\frac{1}{z}\int_{0}^{z}\hat{h}(w+x(s)+\upsilon([0,s]))ds\mu^{*}(dw,dx,d\upsilon)\right]\leq J_E^*. 
\end{aligned}
\end{equation}
Due to the fact that $\hat{h}_N(w)\uparrow\hat{h}(w)$ as $N\rightarrow\infty$ for all $w\in\mathbb{R}_+^I$, the monotone convergence theorem and \eqref{eq:BndErgCostMeasBndUnifN} then give
\begin{align}
&E^*\left[\int_{\R_+^I \times \sD^I \times \sM^I}\hat{h}(w+\upsilon(\{0\}))\mu^{*}(dw,dx,d\upsilon)\right] \nonumber\\
&=\lim_{N\rightarrow\infty}E^*\left[\int_{\R_+^I \times \sD^I \times \sM^I}\hat{h}_{N}(w+\upsilon(\{0\}))\mu^{*}(dw,dx,d\upsilon)\right] \le J_E^*. \label{eq:147}
\end{align}
Due to Theorem \ref{thm:conv}, $P^*$-a.s., $\mathbf{u}(\cdot) = \upsilon((0,\cdot])$ is an admissible control as in Definition \ref{def:BCP2} with $(\R_+^I \times \sD^I \times \sM^I,\mathcal{B}(\R_+^I \times \sD^I \times \sM^I), \mu^*, \{\sF^*(t)\},x, \mathbf{w})$ in place of  $(\tilde{\Omega}, \tilde{\sF}, \tilde{P}, \{\tilde{\sF}(t)\},\tilde{X}, \tilde{W})$, where $\mathbf{w}(\cdot) = w + x(\cdot) + \upsilon([0, \cdot])$. 
So, from \eqref{eq:147}, $P^*$-a.s. we have 
\begin{align*}
\tilde{J}_E^{BCP,*} &\leq   \int_{\R_+^I \times \sD^I \times \sM^I}\hat{h}(\mathbf{w}(0))\mu^*(dw, dx, d\upsilon)   \\
& = \int_{\R_+^I \times \sD^I \times \sM^I}\hat{h}(w+\upsilon(\{0\}))\mu^*(dw, dx, d\upsilon) \le J_E^*, 
\end{align*}
completing the proof.  \hfill \qed


\section{Proofs of Auxiliary Results}\label{sec:MovedProofs}

We finish with the proofs of several results used above. Section \ref{sec:proofNextArrTimeBnd} contains the proof of Proposition  \ref{prop:nextArrTimeBnd},  and the proofs of Propositions   \ref{prop:ArrServBnds} and  \ref{prop:FuildScaledAllocConv} are given in Sections \ref{sec:ProofPropArrServBnds} and \ref{sec:Proofof PropFluidScaledAllocConv}, respectively. 

\subsection{Proof of Proposition \ref{prop:nextArrTimeBnd}}
\label{sec:proofNextArrTimeBnd}
Recall the convention that $u^r_j(0)=0$ for all $r\in\mathbb{N}$ and $j\in\AAA_J$.  Let $j \in \AAA_J$ and $t\in [0,\infty)$ be arbitrary.  
For $n \in \N$, we have
\begin{align*}
\left\{ \tau^{r,A}_j(t)=n\right\}&=\left\{ \left(tr^{2}-u^r_j(n)\right)^{+}\leq \sum_{l=1}^{n-1}u^{r}_{j}(l) < tr^{2} \right\}\\
&=\left\{\tau^{r,A}_{j}\left(\left(t-\frac{u^r_j(n)}{r^{2}}\right)^{+}\right)\leq n-1\right\}\cap \left\{\tau^{r,A}_{j}(t)> n-1\right\}, 
\end{align*}
and note that, for any fixed $x\in [0,\infty)$, the event
\begin{eqnarray*}
\left\{\tau^{r,A}_{j}\left(\left(t-\frac{x}{r^{2}}\right)^{+}\right)\leq n-1<\tau^{r,A}_{j}(t)\right\} 
\end{eqnarray*}
is independent of $u_j^r(n)$. Moreover, by standard results about random walks, $\tau_j^{r,A}(t) < \infty$ a.s.
%
Then, recalling that the $\{u^{r}_{j}(l)\}_{l=1}^{\infty}$ are i.i.d.,  and denoting by $\gamma^r_j$ the distribution of $u_j^r(1)$,
\begin{align}
E\left[u^{r}_{j}(\tau^{r,A}_{j}(t))\right]&=\sum_{n=1}^{\infty}E\left[u^{r}_{j}(n)\mathcal{I}_{\{ \tau^{r,A}_{j}(t)=n \}}\right] \nonumber\\
&=\sum_{n=1}^{\infty}E\left[u^{r}_{j}(n)\mathcal{I}_{\left\{\tau^{r,A}_{j}\left(\left(t-r^{-2}u^r_j(n)\right)^{+}\right)\leq n-1<\tau^{r,A}_{j}(t)\right\}}\right] \nonumber\\
&=\sum_{n=1}^{\infty}\int_0^\infty xE\left[\mathcal{I}_{\left\{\tau^{r,A}_{j}\left(\left(t-r^{-2}x\right)^{+}\right)\leq n-1<\tau^{r,A}_{j}(t)\right\}}\right] \gamma^r_j(dx) \nonumber\\
&=\int_0^\infty xE\left[\sum_{n=1}^{\infty}\mathcal{I}_{\left\{\tau^{r,A}_{j}\left(\left(t-r^{-2}x\right)^{+}\right)\leq n-1<\tau^{r,A}_{j}(t)\right\}}\right] \gamma^r_j(dx) \nonumber\\
&=\int_0^\infty xE\left[\tau^{r,A}_{j}(t)-\tau^{r,A}_{j}\left(\left(t-\frac{x}{r^2}\right)^{+}\right)\right]\gamma^r_j(dx).  \label{eq:uTauBnd1}
\end{align}
Note that 
\[
\tau^{r,A}_{j}(t)-\tau^{r,A}_{j}\left(\left(t-\frac{x}{r^2}\right)^{+}\right) \le  
\min\left\{n\geq 1:\sum_{l=\tau^{r,A}_{j}\left(\left(t-r^{-2}x\right)^{+}\right)+1}^{\tau^{r,A}_{j}\left(\left(t-r^{-2}x\right)^{+}\right)+n}u^{r}_{j}(l)\geq xr^{-2}\right\}
\overset{d}{=} \tau^{r,A}_{j}\left(\frac{x}{r^2}\right).
\]
Consequently, from \eqref{eq:uTauBnd1} we have
\begin{equation}
\label{eq:uTauBnd2}
\sup_{t>0} E\left[u^{r}_{j}(\tau^{r,A}_{j}(t))\right]\leq \int_0^\infty xE\left[ \tau^{r,A}_{j}\left(\frac{x}{r^2}\right)\right]\gamma^r_j(dx).
\end{equation}
Now, due to Condition \eqref{cond:uniformintegrability}, there exists $K \in [0,\infty)$ such that
\begin{equation}
\label{eqn:uniformTailBnd}
\sup_{r \ge 0} E\left[ u^{r}_{j}(1) \mathcal{I}_{\{u^{r}_{j}(1) > K\}} \right] \leq \frac{1}{4\alpha_{j}}, 
\end{equation}
and without loss of generality, we may suppose that $K > 1/(4\alpha_j)$. 
Let $R_0>0$ be such that, for all $r\ge R_0$, $\frac{3}{4}\alpha^{-1}_j \leq (\alpha^r_j)^{-1}$. Then, for $r\ge R_0$, 
\begin{equation*} 
E\left[ u^{r}_{j}(1) \mathcal{I}_{\{u^{r}_{j}(1)\leq K\}} \right] = \frac{1}{\alpha^r_{j}}-E\left[ u^{r}_{j}(1) \mathcal{I}_{\{u^{r}_{j}(1) > K\}} \right]\geq \frac{3}{4\alpha_j}-\frac{1}{4\alpha_j} = \frac{1}{2\alpha_j}, 
\end{equation*}
as well as
\begin{align*}
E\left[ u^{r}_{j}(1) \mathcal{I}_{\{u^{r}_{j}(1)\leq K\}} \right] 
& \leq \frac{1}{4\alpha_{j}}P\left(u^{r}_{j}(1) < \frac{1}{4\alpha_{j}}\right)+KP\left(\frac{1}{4\alpha_{i}} \leq u^{r}_{j}(1) \leq K\right) \\
&\leq \frac{1}{4\alpha_{j}}+KP\left(u^{r}_{j}(1) \ge \frac{1}{4\alpha_{j}} \right).
\end{align*}
Therefore, for $r\ge R_0$  we have
\begin{equation}
P\left(u^{r}_{j}(1) \ge \frac{1}{4\alpha_{j}} \right) \geq \frac{1}{K}E\left[ u^{r}_{j}(1) \mathcal{I}_{\{u^{r}_{j}(1)\leq K\}} \right]-\frac{1}{4K \alpha_{j}} \geq\frac{1}{2K \alpha_{j}}- \frac{1}{4K \alpha_{j}}\geq \frac{1}{4K \alpha_{j}}.\label{eq:530}
\end{equation}
Now, define  
\[
	C^{r}_{j}(n)=\sum_{l=1}^{n}\mathcal{I}_{\{ u^{r}_{j}(l)\geq (4 \alpha_{j})^{-1} \}} \qquad  \mbox{and} \qquad  \zeta^{r}_{j}(x)=\min\{n\geq0:C^{r}_{j}(n)= \lceil 4x \alpha_{j}\rceil \}. 
\]
For some $m \in \N$, suppose that $\zeta_j^r(x) \le m$. Then, 
\[
	\sum_{l=1}^m u_j^r(l) \ge \frac{C_j^r(m)}{4\alpha_j} \ge \frac{4x\alpha_j}{4\alpha_j}= x, 
\] 
or $\tau_j^{r,A}(r^{-2}x) \le m$. 
Therefore, $E\left[\tau^{r,A}_{j}\left(r^{-2}x\right) \right] \leq E\left[ \zeta^{r}_{j}(x)\right]$.  Because the $\{u^{r}_{j}(l)\}_{l=1}^{\infty}$ are i.i.d., from \eqref{eq:530} we see that $\zeta^{r}_{j}(x)$ is the sum of $\lceil 4x \alpha_{j}\rceil$ independent geometric distributions with probability of success $P(u_j^r(1) \ge (4\alpha_j)^{-1})\geq (4K \alpha_{j})^{-1}$.  Thus, for all $r\ge R_0$, we have 
\begin{equation*}
E\left[\tau^{r,A}_{j}\left(\frac{x}{r^2}\right) \right]\leq E\left[ \zeta^{r}_{j}(x) \right] \leq  (1+ 4x \alpha_{j}) 4K \alpha_{j} \leq 4K \alpha_{j} + 16xK \alpha^{2}_{j}.
\end{equation*}
Combining this with \eqref{eq:uTauBnd2} implies that for $r\ge R_0$  we have
\begin{align*}
\sup_{t>0}E\left[u^{r}_{j}(\tau^{r,A}_{j}(t))\right] &\leq \int_0^\infty x(4K \alpha_{j} + 16xK \alpha^{2}_{j}) \gamma^r_j(dx) \\
&= \frac{4K \alpha_{j}}{\alpha^{r}_{j}}+16\left((\sigma^{u,r}_{j})^2 + \frac{1}{(\alpha_j^r)^2}\right) K \alpha^{2}_{j}. 
\end{align*}
The result now follows from the uniform integrability assumed in Condition \ref{cond:uniformintegrability}.\hfill \qed

\subsection{Proof of Proposition \ref{prop:ArrServBnds}}
\label{sec:ProofPropArrServBnds}

For (i), note that since $\tau_j^{r,A}(t)$ is the first $n$ such that $\sum_{l=1}^n u_j^r(l) \ge r^2t$, 
\begin{align*}	
	\hat{\Upsilon}_j^{A,r}(t) &= r\left( \frac{1}{r^2} \sum_{l=1}^{\tau_j^{r,A}(t)} u_j^r(l) - t \right) < \frac{u_j^r(\tau_j^{r,A}(t))}{r}. 
\end{align*}
Therefore, Proposition \ref{prop:nextArrTimeBnd} gives 
\[
	\limsup_{k \to \infty} \sup_{t\ge0} P\left(\hat{\Upsilon}_j^{A,k}(t) > \eps \right) \le \limsup_{k \to \infty} \frac{1}{r_k\epsilon}\sup_{t\ge0}  E\left[ u_j^{r_k}(\tau_j^{{r_k},A}(t)) \right] = 0. 
\]


The proof of (ii) is more complicated because it involves the allocation policy $B^{k}(t)$. 
 Let $j\in\AAA_J$ and $\epsilon>0$ be arbitrary.  For $k\in\mathbb{N}$ define the event
\[
\mathcal{H}^{t,k}=\left\{ \bar{\Upsilon}^{A,k}_{j}(t)\leq\frac{1}{4},\sup_{0\leq s \leq 1}\left|\frac{1}{r_{k}} A_{j}^{k,t}(r_{k}^{2}s)-r_{k} s\alpha^{r_{k}}_{j}\right|\leq \frac{r_{k}\alpha^{r_{k}}_{j}}{4}\right\}.
\]
Due to Proposition \ref{prop:nextArrTimeBnd}, Lemma \ref{CLTLLN} (i) and Lemma \ref{lem:sameDist}, 
we have
\begin{equation}
\label{eq:HsetProbBnd}
\lim_{k\rightarrow\infty}\sup_{t\geq 0}P(\mathcal{H}^{t,k})^c=0.
\end{equation}
For $t\geq 0$, $k\geq K$, and $\epsilon > 0$ define the event
\[
\mathcal{C}_{\epsilon}^{t,k}=\left\{\sup_{ x,y \in \left[0,r_k^{-1}\epsilon+L\right],|x-y|\leq r_k^{-1}\epsilon}\left|\hat{S}^{k,t}_{j}(x)-\hat{S}^{k,t}_{j}(y)\right|\leq \epsilon\right\}, 
\]
where we recall $L = \max_i C_i$ from \eqref{eq:CLip}. 

Due to Lemma \ref{CLTLLN} (i) combined with Lemma \ref{lem:sameDist}, for fixed $\epsilon>0$ we have 
\begin{equation}
\label{eq:CsetProbBnd}
\lim_{k\rightarrow\infty}\sup_{t\geq 0}P(\mathcal{C}_{\epsilon}^{t,k})^c=0.
\end{equation}

Let $\beta_0 \doteq \inf_{j \in \mathbb{A}_J} \inf_{r>0} \beta_j^r$ and let $\epsilon_0 \doteq \beta_0\epsilon$.
Note that, for $t>1$, 
\begin{align}
P\left(r_{k}\bar{\Upsilon}^{S,k}_{j}(t)>\epsilon \right)&\leq P\left(\{\tau_{j}^{k,S}(t-1)<\tau_{j}^{k,S}(t)\}\cap\{r_{k}\bar{\Upsilon}^{S,k}_{j}(t)>\epsilon\} \right)+P\left(\tau_{j}^{k,S}(t-1)=\tau_{j}^{k,S}(t) \right).\\
\label{eq:847}
\end{align}
We will now argue that, for $t>1$,
\begin{equation}
\label{eq:UpsilonSGreater1}
P\left(\{r_{k}\bar{\Upsilon}^{S,k}_{j}(t)>\epsilon\}\cap\{\tau_{j}^{k,S}(t-1)<\tau_{j}^{k,S}(t)\}\right)\leq P( \mathcal{C}_{\epsilon_0}^{t-1,k})^c.
\end{equation}

For that note that, for $t>1$, on the event $\{\tau_{j}^{k,S}(t-1)<\tau_{j}^{k,S}(t)\}$, we have $\bar{\xi}_{j}^{S,k}(t-1) < \bar{B}_{j}^{k}(t)$ and 
\begin{equation*}
\bar B_j^{k}(t)= \bar B_j^{k}(t-1) + (\bar B_j^{k}(t)-\bar B_j^{k}(t-1))\leq \bar{\xi}_{j}^{S,k}(t-1) +L
\end{equation*}
which gives $\bar{\xi}_{j}^{S,k}(t-1) < \bar{B}_{j}^{k}(t)\leq \bar{\xi}_{j}^{S,k}(t-1) +L$.
In addition, for $s\in [0, \bar{\Upsilon}^{S,k}_{j}(t))$, $\bar B_j^{k}(t) + s < \bar \xi_j^{S,k}(t)$, and so on this event, 
for $s\in [0, \bar{\Upsilon}^{S,k}_{j}(t))$,
we have 
\begin{align*}
&S^{k,t-1}_{j}\left(r_k^2\left(s+\bar{B}^{k}_{j}(t)-\bar{\xi}^{S,k}_{j}(t-1)\right)\right)-S^{k,t-1}_{j}\left(r_k^2\left(\bar{B}^{k}_{j}(t)-\bar{\xi}^{S,k}_{j}(t-1)\right)\right)\\
&=S^{k}_{j}\left(r_k^2\left(s+\bar{B}^{k}_{j}(t)\right)\right)-S^{k}_{j}\left(r_k^2\left(\bar{B}^{k}_{j}(t)\right)\right)=0
\end{align*}
which gives
\begin{equation*}
\hat{S}^{k,t-1}_{j}\left(s+\bar{B}^{k}_{j}(t)-\bar{\xi}^{S,k}_{j}(t-1)\right)-\hat{S}^{k,t-1}_{j}\left(\bar{B}^{k}_{j}(t)-\bar{\xi}^{S,k}_{j}(t-1)\right)=-r_k\beta_{j}^{k}s.
\end{equation*}
So,  replacing the $s$ above with $\eps/r_k$, we see that, on the set $\{r_{k}\bar{\Upsilon}^{S,k}_{j}(t)>\epsilon\} \cap \{\tau_{j}^{k,S}(t-1)<\tau_{j}^{k,S}(t)\}$,
\begin{align*}
\epsilon_0\le \epsilon\beta^{k}_{j}&\leq \sup_{s\in \left[0,r_k^{-1}\epsilon\right]}\left|\hat{S}^{k,t-1}_{j}\left(s+\bar{B}^{k}_{j}(t)-\bar{\xi}^{S,k}_{j}(t-1)\right)-\hat{S}^{k,t-1}_{j}\left(\bar{B}^{k}_{j}(t)-\bar{\xi}^{S,k}_{j}(t-1)\right)\right| \\
&\leq\sup_{ x,y \in \left[0,r_k^{-1}\epsilon+L\right],|x-y|\leq r_k^{-1}\epsilon}\left|\hat{S}^{k,t-1}_{j}(x)-\hat{S}^{k,t-1}_{j}(y)\right|.
\end{align*}
This proves \eqref{eq:UpsilonSGreater1} and so from \eqref{eq:847} we obtain that
\begin{align}
P\left(r_{k}\bar{\Upsilon}^{S,k}_{j}(t)>\epsilon \right)&\leq
P(\mathcal{C}_{\epsilon_0}^{t-1,k})^c +P(\mathcal{H}^{t-1,k})^c+P\left(\{\tau_{j}^{k,S}(t-1)=\tau_{j}^{k,S}(t)\}\cap\mathcal{H}^{t-1,k} \right).\nonumber\\ 
\label{eq:847b}
\end{align}
Now we estimate the third probability on the right side.
On the event $\{\tau_{j}^{k,S}(t-1)=\tau_{j}^{k,S}(t)\}$, we have
\begin{equation}\label{eq:Sequality}
S^{k}_{j}\left(r_{k}^{2}\bar{B}^{k}_{j}(t)\right)=S^{k}_{j}\left(r_{k}^{2}\bar{B}^{k}_{j}(t-1)\right). 
\end{equation}
Also, since for any $s\geq 0$ 
\begin{equation*}
A^{k}_j \left(r^2_k (s+t-1)\right)=A^{k,t-1}_j\left(r_k^2(s- \bar{\Upsilon}^{A,k}_{j}(t-1))^+\right)+ A_j^{k}\left(r^2_k (t-1)\right)+\mathcal{I}_{\{s\geq  \bar{\Upsilon}^{A,k}_{j}(t-1)>0\}}, 
\end{equation*}
on $\mathcal{H}^{t-1,k}$ we also have
\begin{align*}
\frac{1}{r_{k}}A^{k}_{j}\left(r_k^2(t-1/4)\right) &\geq \frac{1}{r_k}A_j^{k,t-1}\left(r_k^2(3/4 - \bar \Upsilon^{A,k}(t-1))\right) + \frac{1}{r_k} A_j^{k}\left(r_k^2(t-1)\right) 
  \\
&\ge \frac{1}{r_k} A_j^{k}\left(r_k^2(t-1)\right) + r_k\alpha_j^{r_k} (3/4 - \bar \Upsilon^{A,r_k}(t-1)) - \frac{r_k\alpha_j^{r_k}}{4} \\
&\geq \frac{1}{r_{k}}A^{k}_{j}(r_k^2(t-1))+\frac{r_{k}\alpha^{k}_{j}}{4}. 
\end{align*}
This and \eqref{eq:Sequality} give, on the event $\{\tau_{j}^{k,S}(t-1)=\tau_{j}^{k,S}(t)\}\cap\mathcal{H}^{t-1,k}$,
\[
\hat{Q}_{j}^{k}(s)\geq \hat{Q}_{j}^{k}(t-1)+\frac{r_{k}\alpha^{k}_{j}}{4}\geq \frac{r_{k}\alpha^{k}_{j}}{4}, \qquad s \in \left[t - 1/4, t\right]. 
\]
Therefore,  on this event  we have
\[
\int_{t-1}^{t}\hat{Q}_{j}^{k}(s)ds\geq \int_{t-1/4}^{t}\hat{Q}_{j}^{k}(s)ds\geq \frac{r_{k}\alpha^{k}_{j}}{16}, 
\]
from which it follows that, for any $t > 1$, 
\[
P\left(\{\tau_{j}^{k,S}(t-1)=\tau_{j}^{k,S}(t)\} \cap \mathcal{H}^{t-1,k}\right)\leq\frac{16E\left[ \int_{t-1}^{t}\hat{Q}_{j}^{k}(s)ds \right]}{ r_{k}\alpha^{k}_{j}}\leq \frac{16\int_{t-1}^{t}E\left[h\cdot\hat{Q}^{k}(s) \right]ds}{r_k \alpha^{k}_{j} \min_{1 \le l \le J} h_l}.
\]
Consequently, recalling \eqref{eq:847b}, for all $t> 1$ we have, with 
$\kappa \doteq 16 \sup_{r>0, j \in \mathbb{A}_J}(\alpha^{r}_{j})^{-1}(\min_{1 \le l \le J} h_l)^{-1}$
\begin{align*}
P\left(r_{k}\bar{\Upsilon}^{S,k}_{j}(t)>\epsilon \right)&\leq 
 P(\mathcal{C}_{\epsilon_0}^{t-1,k})^c +P(\mathcal{H}^{t-1,k})^c+\frac{\kappa}{ r_k} \int_{t-1}^{t}E\left[h\cdot\hat{Q}^{k}(s) \right]ds, 
\end{align*}
which
gives
\begin{align*}
&\frac{1}{T_{k}}\int_{0}^{T_{k}}P\left(r_{k}\bar{\Upsilon}^{S,k}_{j}(t)>\epsilon\right)dt \\
&\leq\frac{1}{T_{k}}+\frac{1}{T_{k}}\int_{1}^{T_{k}}\left(P(\mathcal{C}_{\epsilon_0}^{t-1,k})^c 
+P(\mathcal{H}^{t-1,k})^c\right)dt 
+ \frac{\kappa}{ r_k} \cdot \frac{1}{T_k}\int_{1}^{T_{k}}\int_{t-1}^{t}E\left[h\cdot\hat{Q}^{k}(s)\right]dsdt\\
&\leq \frac{1}{T_k} + \sup_{t\geq 0}P(\mathcal{C}_{\epsilon_0}^{t,k})^c +\sup_{t\geq 0}P(\mathcal{H}^{t,k})^c 
+ \frac{\kappa}{ r_k} \cdot \frac{1}{T_k}\int_{0}^{T_{k}}E\left[h\cdot\hat{Q}^{k}(t)\right] dt\\
&\leq\frac{1}{T_k} + \sup_{t\geq 0}P(\mathcal{C}_{\epsilon_0}^{t,k})^c +\sup_{t\geq 0}P(\mathcal{H}^{t,k})^c+
\frac{\kappa}{ r_k}\left(J^{k}_E(B^{k})+\frac{1}{r_k}\right). 
\end{align*}
Equations \eqref{eq:HsetProbBnd} and \eqref{eq:CsetProbBnd} show that the right side of the above inequality goes to $0$ as $k\rightarrow \infty$, which, since $j\in \AAA_J$ and $\epsilon>0$ were arbitrary, completes the proof of (ii).
\hfill \qed

\subsection{Proof of Proposition \ref{prop:FuildScaledAllocConv}}
\label{sec:Proofof PropFluidScaledAllocConv}
Part (i) follows directly from Proposition \ref{prop:ArrServBnds} (i) and 
\[
	\left| \left(s - \bar\Upsilon^{A,k}(t)\right)^+ - s \right| \le \left| \bar\Upsilon^{A,k}(t) \right|. 
\]
Now we prove part (ii). 
Fix $u\ge0$, and 
for each $j\in\AAA_J$ and $\epsilon>0$ define
\begin{align*}
\mathcal{U}^{j,k}_{\epsilon,t}&=\left\{ \inf_{0\leq s \leq u }\left(\frac{1}{\rho_{j}}\left(\bar{B}_{j}^{k}(s+t)-\bar{\xi}_j^{S,k}(t)\right)^{+}-s\right) < -\epsilon\right\}, \\
\mathcal{O}^{j,k}_{\epsilon,t}&=\left\{ \sup_{0\leq s \leq u }\left(\frac{1}{\rho_{j}}\left(\bar{B}_{j}^{k}(s+t)-\bar{\xi}_j^{S,k}(t)\right)^{+}-s\right) > \epsilon\right\}, 
\end{align*}
and
\begin{align*}
\mathcal{E}^{j,k}_{\epsilon,t}&=\left\{\sup_{0\leq s \leq u}\left|\frac{1}{r_{k}} A_{j}^{k,t}(r_{k}^{2}s)-r_{k}s\alpha^{k}_{j}\right| \leq \frac{\epsilon}{24}\alpha_{j}  r_{k}\right\}\cap\left\{\sup_{0\leq s \leq Lu}\left|\frac{1}{r_{k}} S_{j}^{k,t}(r_{k}^{2}s)-r_{k}s\beta^{k}_{j}\right| \leq  \frac{\epsilon}{24}\alpha_{j}  r_{k}\right\} \\
&\quad \cap\left\{\sup_{0\leq s \leq u}\left|\left(s-\bar{\Upsilon}_{j}^{A,k}(t)\right)^{+}-s\right| \leq \frac{\epsilon}{24} \right\}. 
\end{align*}
Note that with these definitions, 
\begin{equation}
\label{eq:residServTimeEqOU}
\left\{\sup_{s\in [0,u]}\left|\left(\bar{B}_j^{k}(s+t)-\bar{\xi}_j^{S,k}(u)\right)^+-\rho_j s\right|>\epsilon\right\}=\mathcal{U}^{j,k}_{\rho_j^{-1}\epsilon,t}\cup \mathcal{O}^{j,k}_{\rho_j^{-1}\epsilon,t}, 
\end{equation}
and due to  Lemma \ref{CLTLLN} (i) combined with Lemma \ref{lem:sameDist} and part (i) proved above, we have
\begin{equation}
\label{eq:badBehavSetAllocConvProof}
\lim_{k\rightarrow\infty}\sup_{t\geq 0}P (\mathcal{E}^{j,k}_{\epsilon,t})^c=0
\end{equation}
 for all $j\in \AAA_J$.

  The idea behind the proof is as follows. On the event $\mathcal{E}^{j,k}_{\epsilon,t}$, an under-allocation of resource $j$ (the event $\mathcal{U}^{j,k}_{\epsilon,t}$) implies that the scaled queue length, $\hat{Q}^r_j$, gets very large. So, we can show that the (time averaged) probability of $\mathcal{U}^{j,k}_{\epsilon,t}$ goes to $0$ using the fact that $\{B^{k}\}$ is an asymptotically optimal sequence of controls along with \eqref{eq:badBehavSetAllocConvProof}.  We then use this along with the fact that an over-allocation in resource $j$ (the event $\mathcal{O}^{j,k}_{\epsilon,t}$) implies an underallocation in another (see the resource constaint in Definition \ref{def:admissible} (ii)) to show that the (time averaged) probability of $\mathcal{O}^{j,k}_{\epsilon,t}$ goes to $0$.  These two results combined with \eqref{eq:residServTimeEqOU} complete the proof. 

We first show that 
\begin{equation}
\label{eq:residSerTimeUtoZero}
\lim_{k\rightarrow\infty}\frac{1}{T_{k}}\int_{0}^{T_{k}}P\left( \mathcal{U}^{j,k}_{\epsilon,t}\right)dt=0
\end{equation}
for all $j\in\AAA_J$ and $\epsilon > 0$.
When $u = 0$, $\mathcal{U}^{j,k}_{\epsilon,t} = \{(\bar B_j^{k}(t) - \bar\xi_j^{S,k}(t))^+ < -\epsilon\}  = \emptyset$, so we may assume that $u > 0$, and without loss of generality that $\epsilon \in (0,u)$. 
Also assume that $k$ is sufficiently large that $1/r_k^2\leq \alpha_j\epsilon/8$ and that, for all $j\in \AAA_J$, 
\begin{equation}
\label{eq:alphaRcompAlphaLargeK}
\left(1 - \frac{\epsilon}{12u}\right)\alpha_j \le \alpha^{k}_{j} \le 2\alpha_j, 
\end{equation}
and
\begin{equation}
\label{eq:betaRcompBetaLargeK}
\beta^{k}_{j}\rho_{j} = \beta^{k}_{j} \frac{\alpha_j}{\beta_j}\leq \left(1 + \frac{\epsilon}{12u}\right)\alpha_{j}.
\end{equation}
For any $t,s\geq0$, from \eqref{eq:328},
\begin{align}
\hat{Q}^{k}_{j}(s+t)=&\text{ }\hat{Q}^{k}_{j}(t)+\frac{1}{r_{k}}A_{j}^{k,t}\left(r_{k}^{2}\left(s-\bar{\Upsilon}^{A,k}_{j}(t)\right)^{+}\right)-\frac{1}{r_{k}}S_{j}^{k,t}\left(r_{k}^{2}\left(\bar{B}_{j}^{k}(s+t)-\bar{\xi}_{j}^{S,k}(t)\right)^{+}\right)\nonumber \\
&+\frac{1}{r_{k}}\mathcal{I}_{\left\{s\geq \bar{\Upsilon}^{A,k}_{j}(t)>0\right\}}-\frac{1}{r_{k}}\mathcal{I}_{\left\{\bar{B}_{j}^{k}(s+t)\geq\bar{\xi}_{j}^{S,k}(t), \bar{\Upsilon}^{S,k}_{j}(t)>0\right\}}, \label{eq:queueLengthDiff}
\end{align}
which, combined with $1/r_k^2\leq \alpha_j\epsilon/8$, gives 
\begin{equation*}
\hat{Q}^{k}_{j}(s+t) \geq \frac{1}{r_{k}}A_{j}^{k,t}\left(r_{k}^{2}\left(s-\bar{\Upsilon}^{A,k}_{j}(t)\right)^{+}\right)-\frac{1}{r_{k}}S_{j}^{k,t}\left(r_{k}^{2}\left(\bar{B}_{j}^{k}(s+t)-\bar{\xi}_{j}^{S,k}(t)\right)^{+}\right) - \frac{\alpha_jr_k\epsilon}{8}. 
\end{equation*}
Note that $\rho_{j}^{-1}(\bar{B}_{j}^{k}(s+t)-\bar{\xi}^{S,r}(t))^{+}-s\geq -\epsilon$ for all $s\in [0,\epsilon]$, so on the set $\mathcal{U}^{j,k}_{\epsilon,t}$, there exists some $s^*\in (\epsilon,u]$ such that 
\begin{equation*}
\frac{1}{\rho_{j}}\left(\bar{B}_{j}^{k}(s^*+t)-\bar{\xi}^{S,r}(t)\right)^{+}-s^*<-\epsilon.
\end{equation*}  
Consequently, on the set $\mathcal{E}^{j,k}_{\epsilon,t}\cap\mathcal{U}^{j,k}_{\epsilon,t}$ we have
\begin{align*}
\hat{Q}^{k}_{j}(s^{*}+t)&\geq \alpha^{k}_{j} r_{k}\left(s^{*}-\bar{\Upsilon}^{A,k}_{j}(t)\right)^{+}-\frac{\epsilon}{24}\alpha_{j}  r_{k}- r_{k}\beta^{k}_{j}\left(\bar{B}_{j}^{k}(s^{*}+t)-\bar{\xi}_{j}^{S,k}(t)\right)^{+}-\frac{\epsilon}{24}\alpha_{j}  r_{k} 
 -\frac{\alpha_jr_k\epsilon}{8} \\
&\geq  \alpha^{k}_{j} r_{k}\left(s^{*} -\frac{\epsilon}{24}\right)- r_{k}\beta^{k}_{j}\rho_{j}(s^{*}-\epsilon)-  \frac{\epsilon}{12}\alpha_{j}r_{k} - \frac{\alpha_jr_k\epsilon}{8} \\
&\ge \left(1 - \frac{\epsilon}{12u}\right)\alpha_jr_k\left(s^*-\frac{\epsilon}{24}\right) - r_k\left(1 + \frac{\epsilon}{12u}\right)\alpha_j(s^*-\epsilon) - \frac{5\alpha_jr_k\epsilon}{24} \\
&= - \frac{\alpha_jr_ks^*\epsilon}{6u} + \frac{9\alpha_jr_k\epsilon}{12} + \frac{25\alpha_jr_k\epsilon^2}{288u} \\
&\ge \frac{7\alpha_jr_k\epsilon}{12}, 
\end{align*}
where third inequality uses \eqref{eq:alphaRcompAlphaLargeK} and \eqref{eq:betaRcompBetaLargeK}, and the last inequality uses $s^* \le u$.
Now consider $s \in [\epsilon/24, s^*]$. 
Note that, on $\mathcal{E}^{j,k}_{\epsilon,t}\cap \mathcal{U}^{j,k}_{\epsilon,t}$, $\bar{\Upsilon}^{A,k}_{j}(t)\leq \frac{\epsilon}{24}$, and so 
using \eqref{eq:queueLengthDiff}, \eqref{eq:alphaRcompAlphaLargeK}, \eqref{eq:betaRcompBetaLargeK}, the fact that $S_j^{k,t}(\cdot)$ and $\bar B_j^{k}(\cdot)$ are increasing, and that on $\mathcal{E}^{j,k}_{\epsilon,t}$
\[
	\alpha_j^{r_k}r_kx - \frac{\alpha_jr_k\epsilon}{24} \le \frac{1}{r_k} A_j^{k,t}(r_k^2x) \le \alpha_j^{r_k}r_kx + \frac{\alpha_jr_k\epsilon}{24}, \qquad x \in [0,u], 
\]
we have, on $\mathcal{E}^{j,k}_{\epsilon,t}\cap \mathcal{U}^{j,k}_{\epsilon,t}$,
\begin{align*}
	\hat{Q}^{k}_{j}(s+t) - \hat{Q}^{k}_{j}(s^{*}+t) &\ge \frac{1}{r_{k}}A_{j}^{k,t}\left(r_{k}^{2}\left(s-\bar{\Upsilon}^{A,k}_{j}(t)\right)^{+}\right)-\frac{1}{r_{k}}A_{j}^{k,t}\left(r_{k}^{2}\left(s^*-\bar{\Upsilon}^{A,k}_{j}(t)\right)^{+}\right) \\
	&\ge \alpha_j^{k}r_k\left(s - \bar{\Upsilon}^{A,k}_{j}(t)\right)^+ - \frac{\alpha_jr_k\epsilon}{24} -  \alpha_j^{k}r_k\left(s^* - \bar{\Upsilon}^{A,k}_{j}(t)\right)^+ - \frac{\alpha_jr_k\epsilon}{24} \\
	&\ge - \alpha_j^{r_k}r_k(s^*-s) - \frac{\alpha_jr_k\epsilon}{12} \\
	&\ge -2\alpha_jr_k(s^*-s) - \frac{\alpha_jr_k\epsilon}{12}.
\end{align*}
It follows that 
\[
	\hat{Q}^{k}_{j}(s+t) \ge  \frac{7\alpha_jr_k\epsilon}{12}-2\alpha_jr_k(s^*-s) - \frac{\alpha_jr_k\epsilon}{12} =  \frac{\alpha_jr_k\epsilon}{2} -2\alpha_jr_k(s^*-s) . 
\]
Then, recalling that $s^*>\epsilon$, for any $s\in \left[s^*-\frac{\epsilon}{8},s^*\right] \subset  [\epsilon/24, s^*]$ and on the event $\mathcal{E}^{j,k}_{\epsilon,t}\cap \mathcal{U}^{j,k}_{\epsilon,t}$, 
we have 
\begin{equation*}
\hat{Q}^{k}_{j}(s+t)  \geq  \frac{\alpha_jr_k\epsilon}{2}  - \frac{\alpha_jr_k\epsilon}{4} = \frac{\alpha_jr_k\epsilon}{4}. 
\end{equation*}
Therefore, 
\[
E\left[\left.\int_{s^*-\frac{\epsilon}{8}}^{s^*}\hat{Q}^{k}_{j}(s+t)ds\right|\mathcal{E}^{j,k}_{\epsilon,t}\cap \mathcal{U}^{j,k}_{\epsilon,t}\right] \geq \frac{r_{k}\alpha_{j}\epsilon^2}{32}, 
\]
and consequently, 
\begin{align*}
E\left[\int_{0}^{u}\hat{Q}^{k}_{j}(s+t)ds\right]  &\geq E\left[\mathcal{I}_{\mathcal{E}^{j,k}_{\epsilon,t}\cap \mathcal{U}^{j,k}_{\epsilon,t}}\int_{s^*-\frac{\epsilon}{8}}^{s^*}\hat{Q}^{k}_{j}(s+t)ds\right] \\
&=  E\left[\mathcal{I}_{\mathcal{E}^{j,k}_{\epsilon,t}\cap \mathcal{U}^{j,k}_{\epsilon,t}}E\left[\left.\int_{s^*-\frac{\epsilon}{8}}^{s^*}\hat{Q}^{k}_{j}(s+t)ds\right|\mathcal{E}^{j,k}_{\epsilon,t}\cap \mathcal{U}^{j,k}_{\epsilon,t}\right]\right] \\
 &\geq P\left(\mathcal{E}^{j,k}_{\epsilon,t}\cap \mathcal{U}^{j,k}_{\epsilon,t}\right)\cdot \frac{r_{k}\alpha_{j}\epsilon^2}{32}, 
\end{align*}
which says that
\begin{equation*}
P\left(\mathcal{E}^{j,k}_{\epsilon,t}\cap \mathcal{U}^{j,k}_{\epsilon,t}\right) \leq \frac{32 }{r_{k}\alpha_{j}\epsilon^2}E\left[\int_{0}^{u}\hat{Q}^{k}_{j}(s+t)ds\right].
\end{equation*}
Therefore,
\begin{align*}
\frac{1}{T_{k}}\int_{0}^{T_{k}}P\left( \mathcal{U}^{j,k}_{\epsilon,t}\right)dt&\leq \frac{1}{T_{k}}\int_{0}^{T_{k}}\left(P\left( (\mathcal{E}^{j,k}_{\epsilon,t})^c\right) + P\left(\mathcal{E}^{j,k}_{\epsilon,t}\cap \mathcal{U}^{j,k}_{\epsilon,t}\right)\right)dt\\
&\leq \sup_{t\geq 0}P\left( (\mathcal{E}^{j,k}_{\epsilon,t})^c\right) +\frac{32}{r_{k}\alpha_{j}\epsilon^2}\cdot  \frac{1}{T_{k}}\int_{0}^{T_{k}}E\left[\int_{0}^{u}\hat{Q}^{k}_{j}(s+t)ds\right]dt\\
&\leq \sup_{t\geq 0}P\left( (\mathcal{E}^{j,k}_{\epsilon,t})^c\right) + \frac{32}{r_{k}\alpha_{j}\epsilon^2}\cdot \frac{u}{T_{k}} E\left[\int_{0}^{T_k+u}\hat{Q}^{k}_{j}(t)dt\right]\\
&\leq \sup_{t\geq 0}P\left( (\mathcal{E}^{j,k}_{\epsilon,t})^c\right) + \frac{32}{r_{k}\alpha_{j}\epsilon^2}\cdot \frac{u(T_{k}+u)}{T_k}E\left[\frac{1}{T_{k}+u}\int_{0}^{T_k+u}\hat{Q}^{k}_{j}(t)dt\right]. 
\end{align*}
Now \eqref{eq:residSerTimeUtoZero} follows as in the proof of Proposition \ref{prop:ArrServBnds} on using
\eqref{eq:badBehavSetAllocConvProof}, the properties of $(r_k,B^{k},T_k)$ in Definition \ref{def:optSeq}, and   Lemma \ref{lem:hHatBnds}.  

Next, we show that
\begin{equation}
\label{eq:residSerTimeOtoZero}
\lim_{k\rightarrow\infty}\frac{1}{T_{k}}\int_{0}^{T_{k}}P\left( \mathcal{O}^{l,k}_{\epsilon,t}\right)dt=0
\end{equation}
for all $l\in\AAA_J$ and $\eps > 0$. 
Fix $l\in\AAA_J$, so on the set $\mathcal{O}^{l,k}_{\epsilon,t}$ there exists some $s^*\in [0,u]$ such that
\begin{equation}\label{eq:Blpos}
\bar{B}_{l}^{k}(s^*+t)-\bar{B}_{l}^{k}(t)\geq \left(\bar{B}_{l}^{k}(s^*+t)-\bar{\xi}^{S,r}_l(t)\right)^{+} >\rho_l s^*+\rho_l\epsilon.
\end{equation}
Note that there exists some $i\in\AAA_I$ such that $K_{i,l}=1$.  If $\{j:K_{i,j}=1\}\setminus\{l\}=\emptyset$ (i.e., that $l$ is the only job type processed by resource $i$), then $C_{i}=\rho_l$ (see Condition \ref{heavytraffic}) and
\begin{equation*}
\left[K\left(\bar{B}_{l}^{k}(s^*+t)-\bar{B}_{l}^{k}(t)\right)\right]_{i}>\rho_l s^*+\rho_l\epsilon>C_{i}s^*, 
\end{equation*}
which, due to Definition \ref{def:admissible} (ii), cannot occur with positive probability. This means that $\{j:K_{i,j}=1\}\setminus\{l\}=\emptyset$ implies that $P\left(\mathcal{O}^{l,k}_{\epsilon,t}\right)=0$.  So, now assume that $\{j:K_{i,j}=1\}\setminus\{l\}\neq\emptyset$, i.e. that resource $i$ processes job types other than $l$.  Then in order to satsify Definition \ref{def:admissible} (ii),  we must have
\begin{align*}
 C_{i}s^*\geq \sum_{j\in\AAA_J}K_{i,j}\left(\bar{B}_{j}^{k}(s^*+t)-\bar{B}_{j}^{k}(t)\right)> \sum_{j\in\AAA_J\setminus\{l\}}K_{i,j}\left(\bar{B}_{j}^{k}(s^*+t)-\bar{B}_{j}^{k}(t)\right)+\rho_l s^*+\rho_l \epsilon. 
\end{align*}
Since $C_{i}=\sum_{j\in\AAA_J}K_{i,j}\rho_j$, this gives
\begin{equation*}
\sum_{j\in\AAA_J\setminus\{l\}}K_{i,j}\rho_j s^*-\rho_l \epsilon>\sum_{j\in\AAA_J\setminus\{l\}}K_{i,j}\left(\bar{B}_{j}^{k}(s^*+t)-\bar{B}_{j}^{k}(t)\right), 
\end{equation*}
which implies
\begin{align*}
	-\rho_l\epsilon &> \sum_{j \in \AAA_J\setminus\{l\}} K_{i,j}\left(\bar B_j^{k}(s^*+t) - \bar B_j^{k}(t) - \rho_j s^*\right) \\
	&\ge \#\left(\{j : K_{i,j}=1\}\setminus\{l\} \right) \cdot \max_{j \in \AAA_J\setminus\{l\}}\rho_j \cdot \min_{j \in \AAA_J\setminus\{l\}} \left(\frac{1}{\rho_j} \left(\bar B_j^{k}(s^*+t) - \bar B_j^{k}(t)\right) - s^*\right) \\
	&\ge J\cdot  \max_{j \in \AAA_J}\rho_j \cdot \min_{j \in \AAA_J\setminus\{l\}} \left(\frac{1}{\rho_j} \left(\bar B_j^{k}(s^*+t) - \bar B_j^{k}(t)\right) - s^*\right), 
\end{align*}
since $\{j:K_{i,j}=1\}\setminus\{l\}\neq\emptyset$ and we must have
$ (\rho_j^{-1} (\bar B_j^{k}(s^*+t) - \bar B_j^{k}(t)) - s^*) < 0$ for some $j \in \AAA_J\setminus\{l\}$ in order for the inequality to be satisfied.  This gives 
\begin{align*}
\min_{j\in \AAA_J}\left(\frac{1}{\rho_j}\left(\bar{B}_{j}^{k}(s^*+t)-\bar{\xi}_j^{S,r}(t)\right)^+-s^*\right)&\leq\min_{j \in \AAA_J\setminus\{l\}}\left(\frac{1}{\rho_j}\left(\bar{B}_{j}^{k}(s^*+t)-\bar{B}_{j}^{k}(t)\right)-s^*\right) \\
&<-\frac{\rho_l\epsilon}{J\max_j\rho_j} \doteq - \tilde \epsilon.
\end{align*}
Consequently,  
\begin{equation*}
\mathcal{O}^{l,k}_{\epsilon,t}\subset \bigcup_{j\in \AAA_j}\mathcal{U}^{j,k}_{\tilde \epsilon,t}, 
\end{equation*}
and so, from \eqref{eq:residSerTimeUtoZero} we now have that
\begin{equation*}
\lim_{k\to\infty} \frac{1}{T_k}\int_0^{T_k}P\left(\mathcal{O}^{l,k}_{\epsilon,t}\right)\,dt \leq \lim_{k\to\infty}\sum_{j \in \AAA_J}\frac{1}{T_k}\int_0^{T_k}P\left(\mathcal{U}^{j,k}_{\tilde \epsilon,t} \right)\,dt = 0.
\end{equation*}
This establishes \eqref{eq:residSerTimeOtoZero}, which together with \eqref{eq:residServTimeEqOU} and \eqref{eq:residSerTimeUtoZero}
completes the proof of (ii).\\ \ \\ 

\noindent {\bf Acknowledgements}
AB was supported in part by the NSF (DMS-2134107 and DMS-2152577). 




\bibliography{networks}
\bibliographystyle{amsplain}

\vspace{\baselineskip}
%
%
\scriptsize{
\textsc{\noindent A. Budhiraja \newline
Department of Statistics and Operations Research, and\newline
School of Data Science and Society\newline
University of North Carolina\newline
Chapel Hill, NC 27599, USA\newline
email:  budhiraj@email.unc.edu \vspace{\baselineskip} }

\textsc{\noindent M. Conroy\newline
School of Mathematical and Statistical Sciences,\newline
 Clemson University,\newline
  Clemson SC 29634, USA\newline
email: mconroy@clemson.edu \vspace{\baselineskip}}

\textsc{\noindent D. Johnson\newline
Department of Mathematics and Statistics, \newline
Elon University,\newline
 Elon, NC 27244, USA\newline
email: djohnson52@elon.edu}
}

\end{document}